\documentclass[12pt]{article}
\usepackage{graphicx}
\usepackage{amsmath,amsthm,amssymb,enumerate, esint}
\usepackage{color}
\catcode`\@=11 \@addtoreset{equation}{section}

\catcode`\@=12

\newtheorem{Theorem}{Theorem}[section]
\newtheorem{Proposition}{Proposition}[section]
\newtheorem{Lemma}{Lemma}[section]
\newtheorem{Corollary}{Corollary}[section]
\newtheorem{Definition}{Definition}[section]

\theoremstyle{definition}
\newtheorem{Example}[Theorem]{Example}
\newtheorem{Remark}[Theorem]{Remark}

\newcommand{\bTheorem}[1]{
\begin{Theorem} \label{T#1} }
\newcommand{\eT}{\end{Theorem}}

\newcommand{\bProposition}[1]{
\begin{Proposition} \label{P#1}}
\newcommand{\eP}{\end{Proposition}}

\newcommand{\bLemma}[1]{
\begin{Lemma} \label{L#1} }
\newcommand{\eL}{\end{Lemma}}

\newcommand{\bCorollary}[1]{
\begin{Corollary} \label{C#1} }
\newcommand{\eC}{\end{Corollary}}

\newcommand{\bRemark}[1]{
\begin{Remark} \label{R#1} }
\newcommand{\eR}{\end{Remark}}

\newcommand{\bDefinition}[1]{
\begin{Definition} \label{D#1} }
\newcommand{\eD}{\end{Definition}}

\newcommand{\bFormula}[1]{
\begin{equation} \label{#1}}
\newcommand{\eF}{\end{equation}}

\newcommand{\Divh}{{\rm div}_h}

\newcommand{\Ov}[1]{\overline{#1}}

\newcommand{\DC}{C^\infty_c}

\newcommand{\vr}{\varrho}

\newcommand{\vu}{\vc{u}}
\newcommand{\vc}[1]{{\bf #1}}
\newcommand{\Ome}{\Omega_\ep}
\newcommand{\xh}{\vc{x}_h}
\newcommand{\Div}{{\rm div}_x}
\newcommand{\Grad}{\nabla_x}

\newcommand{\tn}[1]{\mathbb{#1}}
\newcommand{\dx}{{\rm d} {x}}
\newcommand{\dz}{{\rm d} {z}}
\newcommand{\dt}{{\rm d} t }

\newcommand{\intO}[1]{\int_{\Omega} #1 \ \dx}

\newcommand{\intOe}[1]{\int_{\Omega_\ep} #1 \ \dx}

\newcommand{\ep}{\varepsilon}

\newcommand{\R}{\mathbb{R}}

\definecolor{Cgrey}{rgb}{0.85,0.85,0.85}
\definecolor{Cblue}{rgb}{0.50,0.85,0.85}
\definecolor{Cred}{rgb}{1,0,0}
\definecolor{fancy}{rgb}{0.10,0.85,0.10}

\newcommand\Cbox[2]{%
    \newbox\contentbox%
    \newbox\bkgdbox%
    \setbox\contentbox\hbox to \hsize{%
        \vtop{
            \kern\columnsep
            \hbox to \hsize{%
                \kern\columnsep%
                \advance\hsize by -2\columnsep%
                \setlength{\textwidth}{\hsize}%
                \vbox{
                    \parskip=\baselineskip
                    \parindent=0bp
                    #2
                }%
                \kern\columnsep%
            }%
            \kern\columnsep%
        }%
    }%
    \setbox\bkgdbox\vbox{
        \color{#1}
        \hrule width  \wd\contentbox %
               height \ht\contentbox %
               depth  \dp\contentbox
        \color{black}
    }%
    \wd\bkgdbox=0bp%
    \vbox{\hbox to \hsize{\box\bkgdbox\box\contentbox}}%
    \vskip\baselineskip%
}

\date{}

\begin{document}


\title{A rigorous justification of the Euler and Navier-Stokes equations with geometric effects}

\author{Peter Bella \and Eduard Feireisl 
\thanks{The research of E.F.~leading to these results has received funding from the
European Research Council under the European Union's Seventh
Framework Programme (FP7/2007-2013)/ ERC Grant Agreement
320078. The Institute of Mathematics of the Academy of Sciences of
the Czech Republic is supported by RVO:67985840.}
\and Marta Lewicka 
\thanks{M.L. was partially supported by the NSF grant DMS-1406730} 
\and Anton\' \i n Novotn\' y} 

\maketitle

\centerline{Max Planck Institute for Mathematics in the Sciences}

\centerline{Inselstrasse 22, 04103 Leipzig, Germany}

\bigskip

\centerline{Institute of Mathematics of the Academy of Sciences of the Czech Republic}

\centerline{\v Zitn\' a 25, CZ-115 67 Praha 1, Czech Republic}

\bigskip

\centerline{University of Pittsburgh, Department of Mathematics}

\centerline{301 Thackeray Hall, Pittsburgh, PA 15260, USA}

\bigskip

\centerline{Institut Math\'ematiques de Toulon, EA2134, University of Toulon}

\centerline{BP 20132, 839 57 La Garde, France }

\bigskip


\begin{abstract}
We derive the 1D isentropic Euler and Navier-Stokes equations describing
the motion of a gas through a nozzle of variable cross section as the
asymptotic limit of the 3D isentropic Navier-Stokes system in a cylinder, the diameter of
which tends to zero. Our method is based on the relative energy inequality
satisfied by any weak solution of the 3D Navier-Stokes system and a variant of Korn-Poincar\' e's inequality
on thin channels that may be of independent interest.
\end{abstract}



\section{Introduction} \label{i}

A simple model of the flow of a compressible gas through a nozzle of
variable cross section describes the evolution of the mass density
$\vr_E = \vr_E(t,z)$ and the velocity $u_E = u_E(t,z)$ by means of the
Euler system: 
\begin{align}
\label{i1}
\partial_t (\vr_E A) + \partial_z (\vr_E u_E A) &= 0,\\
\label{i2}
\partial_t (\vr_E u_E A) + \partial_z (\vr_E u_E^2 A) + A \partial_z p(\vr_E) &= 0,
\end{align}
where $p = p(\vr_E)$ is the pressure and $A = A(z)$ is the 2D measure
of the cross section at the ``vertical'' position $z$, see e.g., LeFloch and Westdickenberg 
\cite{LeFWes}. We also consider a similar model including the effect
of viscosity with an additional drift term, namely: 
\begin{align}
\label{i1a}
\partial_t (\vr_{NS} A) + \partial_z (\vr_{NS} u_{NS} A) &= 0,\\
\label{i2a}
\partial_t (\vr_{NS} u_{NS} A) + \partial_z (\vr_{NS} u_{NS}^2 A) +
A \partial_z p(\vr_{NS}) &= A \Big( \frac{4 \mu}{3} + \eta\Big) \partial^2_z u_{NS} +
A \left( \frac{\mu}{3} + \eta\right) \partial_z \Big(\frac{\partial_z A}{A} u_{NS}\Big). 
\end{align}

The purpose of this paper is to show that (smooth) solutions of the above problems can be
identified as the asymptotic limits of the 3D Navier-Stokes system: 
\begin{align}\label{i3}
\partial_t \vr + \Div (\vr \vu) &= 0,
\\ \label{i4}
\partial_t (\vr \vu) + \Div (\vr \vu \otimes \vu) + \Grad p(\vr) &= \lambda \Div \tn{S}(\Grad \vu),
\end{align}
\bFormula{i5}
\tn{S} (\Grad \vu) = \mu \left( \Grad \vu + \Grad^t \vu - \frac{2}{3}
  \Div \vu \tn{I} \right) + \eta \Div \vu  \tn{I},\qquad \mu > 0, \ \eta \geq 0, 
\eF
considered in the physical domain:
\bFormula{i6}
\begin{split}
\Ome &= \left\{x= (x_1,x_2, z)
 \equiv (\xh,z)\  \Big|\ z \in (0,1), \ \xh \in \ep\omega_h (z) \right\},
\end{split}
\eF
under the slip boundary conditions:
\bFormula{i7}
\vu \cdot \vc{n}|_{\partial \Omega_\ep} = 0, \qquad
 \left[ \tn{S}(\Grad \vu) \cdot \vc{n} \right] \times \vc{n}|_{\partial \Omega_\ep} = 0,
\eF
provided that $\ep \to 0$. Here, $\{ \omega_h(z) \}_{z \in [0,1]}$ is
a family of sufficiently smooth open bounded simply connected subsets
of $\mathbb{R}^2$, with fairly arbitrary geometry (see Section \ref{p}
for details), where we define: 
\[
A(z) := |\omega_h(z)|.
\]

\medskip

Our approach is based on the concept of \emph{dissipative} weak
solutions to the Navier-Stokes system and the associated relative
energy inequality proved in \cite{FeJiNo}, \cite{FeNoSun} 
(cf. also Germain \cite{Ger}). This method provides an explicit rate of convergence in terms of the initial
data and the parameters $\ep$ and $\lambda$.
Namely, we show that the Euler system
(\ref{i1}), (\ref{i2}) is obtained as the inviscid limit of (\ref{i3}--\ref{i7}) when both $\ep$ and the
positive parameter $\lambda$ in (\ref{i4}) tend to zero. Keeping $\lambda = 1$ we obtain the Navier-Stokes
system (\ref{i1a}), (\ref{i2a}). Note that the dependence on the thin
channels $\Omega_\ep$ cross sections $\ep\omega_h(z)$ in the residual
equations (\ref{i1})-(\ref{i2a}), is manifested solely through the
area $A(z)$, and it is independent of the curvature or other
finer properties of the shape of the boundary.
Strangely enough, the asymptotic analysis is more delicate for the Navier-Stokes limit,
where certain quantities must be controlled by means of a variant of
the celebrated \emph{Korn-Poincar\' e inequality}: 
\begin{equation} \label{korn-poin}
\intOe{ |\vc{v} |^2 } \leq C_{KP} \intOe{ \left| \Grad \vc{v} + \Grad^t \vc{v} \right|^2 }
\end{equation}
to be satisfied, with a constant $C_{KP}$ independent of $\ep \to 0$,
for any vector field $\vc{v}$ such that:
\begin{equation}\nonumber 
\begin{aligned}
\vc{v}(x) \cdot \vc{n} &= 0 \qquad &\forall x = (\xh,z) \in \partial \Ome, \ z \in (0,1),\\
\vc{v}(x) &= 0 \qquad &\forall x = (\xh,z) \in \Ov{\Omega}_\ep, \ z \in \{ 0,1 \}.
\end{aligned}
\end{equation}
Note that since we do not attempt to prove the {\it conformal} version of the
Korn-Poincar\' e inequality, specifically:
\begin{equation}\label{CKP}
\intOe{ |\vc{v} |^2 } \leq C_{\textrm{CKP}} \intOe{ \Bigl| \Grad \vc{v} + \Grad^t \vc{v} - \frac{2}{3} \Div \vc{v} \tn{I} \Bigr|^2 },
\end{equation}
we assume that the bulk viscosity $\eta$ is strictly positive. 

\medskip

The paper is organized as follows. In Section \ref{p}, we recall the
concept of dissipative weak solutions to the Navier-Stokes system
(\ref{i3}--\ref{i5}), (\ref{i7}); state and explain the assumption on
the channel-like domains $\Omega_\ep$ 
and the pressure function $p$; and present the main results concerning
the asymptotic limits. In Section \ref{r}, we introduce the relative entropy inequality and derive the
necessary uniform bounds independent of the parameters $\ep$ and
$\lambda$. The asymptotic limits are performed 
in Section \ref{c}. The paper is concluded by the proof of the
Korn-Poincar\' e inequality (\ref{korn-poin}) in Section 5, together with other
related results and problems that may be of independent interest.

\section{Preliminaries and statements of main results} \label{p}

Similarly to the notation $x = (\xh, z)$, the subscript $h$ used in
the differential operators will refer to the horizontal variables. The
pressure $p = p(\vr)$ is assumed to be a function of the density, and to satisfy:
\begin{equation} \label{press}
\begin{split}
p \in C[0, \infty) \cap C^3(0, \infty), &\qquad  \ p(0) = 0, \qquad  p'(\vr) > 0 \quad
\forall  \vr > 0,\\ &
\mbox{and} \quad \lim_{\vr \to \infty} \frac{p'(\vr)}{\vr^{\gamma - 1}} = p_\infty > 0
\quad  \mbox{for a certain}\ \gamma > \frac{3}{2}.
\end{split}
\end{equation}

\begin{Remark}
The assumption for the pressure to be a strictly increasing function of $\vr$ is indispensable for our results.
The growth restriction imposed through the value of $\gamma$ is required by the available existence
theory for the compressible Navier-Stokes system (\ref{i3}--\ref{i5}).
\end{Remark}

Next, we specify our requirements concerning the geometry of the spatial domains $\Ome$ introduced in (\ref{i6}).
As each $\Ome$ is obtained via a simple scaling, it is convenient to
formulate our hypotheses in terms of the basic domain: 
\[
\Omega = \left\{x= (\xh, z) \ \Big|\ z \in (0,1), \ \xh \in \omega_h(z) \right\}.
\]
Namely, we suppose there is a vector field $\vc{V}_h = \vc{V}_h (\xh, z): \Ov{\Omega} \to \mathbb{R}^2$ such that:
\begin{equation} \label{hyp0}
\begin{split}
& \nabla_h {\rm div}_h \vc{V}_h = 0 \quad \mbox{and} \quad \Delta_h \vc{V}_h = 0
\quad \mbox{in } \Omega;\\
& [\vc{V}_h (\xh, z), 1] \in T_{(\xh, z)}(\partial\Omega)\qquad \forall z\in
(0,1), \ \xh\in\partial\omega_h(z).
\end{split}
\end{equation}
The first condition above means that $ {\rm div}_h \vc{V}_h $ depends
only on the variable $z$, while the last condition states that the
vector field $[\vc{V}_h, 1] \in \mathbb{R}^3$
is tangent to $\partial \Omega$ on the lateral boundary $\partial \Omega \cap \{
0 < z < 1\}$.

\begin{Lemma} \label{rO1}
Assume that the lateral boundary of $\Omega$ is of class $C^{r,\alpha}$
with $r\geq 2, \alpha\in (0,1)$. Then:
\begin{itemize}
\item[(i)] There exists a vector field $\vc{V}_h\in C^{r-1,\alpha}(\Ov{\Omega};\R^2)$ satisfying  (\ref{hyp0}).
\item[(ii)] Let $\phi$ be the flow of $\vc{V}_h$, namely:
$\displaystyle{\frac{\mbox{d}}{\mbox{d} t}\phi(\cdot, t) }= \vc{V}_h(\phi(\cdot,
t),t)$ and $\phi(\cdot, 0) = id_{\omega_h(0)}$.
Then: 
$$\omega_h(z) = \Big\{\phi(\vc{x}_h,z) \ \Big| \
\vc{x}_h\in\omega_h(0)\Big\} \qquad \forall z\in[0,1].$$
\item[(iii)] Recalling that $A(z) = |\omega_h(z)|$, there holds
\begin{equation} \label{g1a}
A(z) {\rm div}_h \vc{V}_h(z) = \partial_z A (z).
\end{equation}
\end{itemize}
\end{Lemma} 
\begin{proof}
{\bf 1.} To prove (i), we first define a vector field $\vc{w}_h \in \mathbb{R}^2$ 
on the lateral boundary of $\Omega$, through the following two conditions:
\[
\begin{split}
&\vc{w}_h(\xh, z) \ \mbox{is parallel to the normal vector}\ \vc{n}_h
\ \mbox{to $\omega_h(z)$ at}\ \xh \in \partial \omega_h(z); \\
&\mbox{the vector} \ [ \vc{w}_h(\xh, z), 1 ] \ \mbox{is tangent to}
\ \partial \Omega \ \mbox{at}\ (\xh, z).
\end{split}
\]
Let now $\vc{w}_h=\vc{w}_h(\xh, z)\in\mathbb{R}^2$ be any extension of
$\vc{w}_h$ on $\Ov{\Omega}$, of regularity $C^{r-1, \alpha}$, and denote
$\tilde{\vc{X}} = [\vc{w}_h,1]\in\mathbb{R}^3$ the vector field on
$\Ov{\Omega}$, whose flow $\tilde\Phi$ describes the evolution of the
cross sections $z\mapsto\omega_h(z)$. Namely:
$$\displaystyle{\frac{\mbox{d}}{\mbox{d} t}\tilde\Phi(\cdot, t) }= \tilde{\vc{X}}(\tilde\Phi(\cdot,
t),t), \qquad \tilde\Phi(\cdot, 0) = id_{\omega_h(0)}$$
and we have:
\begin{equation}\label{pome0}
\Big\{\tilde\Phi\big(\vc{x}_h,z) \ \Big| \ \vc{x}_h\in\omega_h(0)\Big\}
= \Big\{\tilde\phi(\vc{x}_h,z) \ \Big| \ \vc{x}_h\in\omega_h(0)\Big\}
\times \{z\} = \omega_h(z)\times\{z\},
\end{equation}
where $\tilde \phi$ is the flow of $\vc{w}_h$, so that:
$$\displaystyle{\frac{\mbox{d}}{\mbox{d} t}\tilde\phi(\cdot, t) }= {\vc{w}_h}(\tilde\phi(\cdot,
t),t), \qquad \tilde\phi(\cdot, 0) = id_{\omega_h(0)}.$$
By a change of variables, we now obtain:
\begin{equation}\label{pome}
\begin{split}
\partial_z A(z) & = \partial_z\Big(\int_{\omega_h(0)} \mbox{det}\nabla_h\tilde\phi(\xh,z)~\mbox{d}\xh\Big)
= \int_{\omega_h(0)} \partial_z\big(\mbox{det}\nabla_h\tilde\phi(\xh,z)\big)~\mbox{d}\xh
\\ & = \int_{\omega_h(0)}
\big(\mbox{det}\nabla_h\tilde\phi(\xh,z)\big)\Big(\mbox{div}_h\vc{w}_h(\tilde\phi(\xh,
z), z)\Big)~\mbox{d}\xh \\ & = \int_{\omega_h(z)} \mbox{div}_h\vc{w}_h(\xh, z)~\mbox{d}\xh = \int_{\partial\omega_h(z)} \vc{w}_h\cdot\vc{n}_h.
\end{split}
\end{equation}

{\bf 2.} Next, we define $U_h= U_h(\xh,z)\in\mathbb{R}$ to be the unique solution of the Neumann problem:
\begin{equation}\label{pome2}
\begin{split}
\Delta_h U_h(\xh,z) = \frac{\partial_z A(z)}{A(z)} \quad \mbox{in}\ \omega_h(z), 
\qquad 
\Grad U_h(\xh,z) \cdot \vc{n}_h = \vc{w}_h (\xh,z) \cdot \vc{n}_h \quad \mbox{on} \ \partial \omega_h(z). 
\end{split}
\end{equation}
This problem has a solution $U_h\in C^{r-1, \alpha}$ enjoying
``horizontal'' regularity $U_h\in C^{r, \alpha}(\omega_h(z))$ because
of the compatibility in: $\int_{\omega_h(z)} \frac{\partial_z A(z)}{A(z)}
~\mbox{d}\xh = \partial_zA(z) = \int_{\partial\omega_h(z)}
\vc{w}_h\cdot\vc{n}_h$, valid in view of (\ref{pome}). 
The desired vector field $\vc{V}_h$ can then be taken as:
\[
\vc{V}_h (\xh, z) = \nabla_h U_h (\xh, z) \quad \mbox{ in } \Omega.
\]
Clearly, $\mbox{div}_h\vc{V}_h = \Delta_hU_h$ is constant in
$\omega_h(z)$ and $\Delta_h\vc{V}_h = \nabla_h\Delta_hU_h=0$ by
(\ref{pome2}). Moreover, on the lateral boundary of $\Omega$, the vector fields 
$\vc{V}_h$ and $\vc{w}_h$ differ by a vector tangent to
$\partial\omega_h(z)$. Therefore $\vc{V}_h$ satisfies (\ref{hyp0}),
which achieves (i). We also automatically obtain (ii), by the same
reasoning as in (\ref{pome0}).
Finally, applying (\ref{pome}) where $\phi$ replaces $\tilde\phi$ and
$\vc{V}_h$ replaces $\vc{w}_h$, we get (iii): 
\begin{equation*}
\partial_z A (z) = \int_{\omega_h(z) } {\rm div}_h \vc{V}_h(\xh, z) \ {\rm d}\xh 
= A(z) {\rm div}_h \vc{V}_h(z).
\end{equation*}

{\bf 3.} To finish the proof, we establish regularity of
the field $\vc{V}_h(\xh,z)$ with respect to the ``vertical'' variable
$z$. To this end, we pull back the boundary problem (\ref{pome2}) to the fixed domain $\omega_h(0)$:
\[
\begin{split}
{\rm div}_h \left( \mathbb{B}(\xh,z) \nabla_h \tilde U_h (\xh,z)
\right) &= \Big(\mbox{det} \nabla_h \tilde \phi (\xh, z) \Big) \frac{\partial_z
  A(z)}{A(z)} \qquad \mbox{in}\ \omega_h(0),\\ 
\nabla_h \tilde{U}_h (\xh, z) \cdot \tilde{{\bf n}}_h (\xh, z) &= \tilde{{\bf
    w}}_h(\xh, z) \cdot {\bf n}_h(\tilde\phi(\xh, z)) \qquad \mbox{on}\ \partial \omega_h(0), 
\end{split} 
\]
where:
\[
\begin{split}
\tilde{U}_h(\xh, z) &= U_h(\tilde{\phi}(\xh, z), z), \qquad \tilde{ \vc{w}
}_h(\xh, z) =  \vc{w}_h (\tilde{\phi} (\xh, z), z) \\
\mathbb{B} (\xh, z) &= \left[\mbox{cof }
  \nabla_h {\tilde{\phi}} (\xh, z) \right]^t
\left[ (\nabla_h {\tilde{\phi}})^{-1}(\xh, z)\right]^t \\ & 
= \Big(\mbox{det} \nabla_h \tilde \phi (\xh, z) \Big)  \left[
  (\nabla_h {\tilde{\phi}})^{-1}(\xh, z) \right]
\left[ (\nabla_h {\tilde{\phi}})^{-1}(\xh, z)\right]^t  \\
\tilde{\bf n}_h(\xh, z) &= \left[ (\nabla_h
  {\tilde{\phi}})^{-1}(\xh, z) \right]  \vc{n}_h (\tilde\phi(\xh, z), z). 
\end{split}
\]
Thus, differentiating with respect to $z$ and using the standard
elliptic estimates we obtain the desired regularity in $z$. 
This ends the proof of Lemma \ref{rO1}.
\end{proof}

\begin{Example}
A typical example of a thin channel that we have in mind is:
$$ \Ome = \left\{ x= (x_1,x_2, z) \equiv ( \xh, z)\  \Big|\ z \in (0,1),
  \ \left| \xh - \ep {X}(z) \right|^2 < R^2(z) \right\}, $$
where $X:[0,1]\to\mathbb{R}^2$ and $R:[0,1]\to(0, +\infty)$ are two
given smooth functions, to the effect that each cross section $\omega_h(z)$ is simply a
circle $B(X(z), R(z))\subset\mathbb{R}^2$.
Note that we can then take:
\begin{equation*}
\vc{V}_{h}(\xh, z) = \frac{\partial_z R(z)}{R(z)} (\xh - {X}(z)) + \partial_z X(z).
\end{equation*}
We also check directly that $A(z) {\rm div}_h \vc{V}_h(z) = \pi R(z)^2 \cdot
2\frac{\partial_zR(z)}{R(z)} = 2\pi R(z)\partial_zR(z) = \partial_z A(z)$.
\end{Example}

\subsection{Dissipative weak solutions to the compressible Navier-Stokes system}

\begin{Definition} \label{D1}
We say that $[\vr, \vu]$ is a (weak) \emph{dissipative solution} to the Navier-Stokes system
(\ref{i3}--\ref{i5}) in the space-time cylinder $(0,T)\times \Ome$ with the boundary conditions (\ref{i7}) if and only if:
\begin{itemize}
\item $
\vr \in C_{\rm weak}([0,T]; L^\gamma(\Ome)),\  \vr \vu \in C_{\rm
  weak}([0,T]; L^\gamma(\Ome; \mathbb{R}^3)), \ \vu \in L^2(0,T;
W^{1,2}(\Ome; \mathbb{R}^3)),$\\ and $
\vr \geq 0 \ \mbox{a.e. in}\ (0,T) \times \Ome, \ \vu \cdot
\vc{n}|_{\partial \Ome} = 0; $
\item For any test function $\varphi \in C^\infty([0,T] \times
  \Ov{\Omega}_\ep)$ there holds:
\begin{equation} \label{mb}
\left[ \intOe{ \vr \varphi } \right]_{t = 0}^{t = \tau} = \int_0^\tau
\intOe{ \left( \vr \partial_t \varphi + \vr \vu \cdot \Grad \varphi \right) } \ \dt;
\end{equation}
\item For any test function $\varphi \in C^\infty([0,T] \times \Ov{\Omega}_\ep; \mathbb{R}^3)$,
$\varphi \cdot \vc{n}|_{\partial \Ome} = 0$ there holds:
\[
\left[ \intOe{ \vr \vu \cdot \varphi } \right]_{t = 0}^{t = \tau} = \int_0^\tau
\intOe{ \left( \vr \vu \cdot \partial_t \varphi + \vr \vu \otimes \vu : \Grad \varphi + p(\vr)
\Div \varphi - \lambda \tn{S}(\Grad \vu): \Grad \varphi \right) } \ \dt;
\]
\item The energy inequality:
\[
\intOe{ \left( \frac{1}{2} \vr |\vu|^2 + H(\vr) \right)(\tau, \cdot) } +
\lambda \int_0^\tau \intOe{ \tn{S}(\Grad \vu) : \Grad \vu } \ \dt
\leq \intOe{ \left( \frac{|\vr \vu |^2 }{2 \vr}  + H(\vr) \right)(0, \cdot) },
\]
with:
\[
H(\vr) = \vr \int_1^\vr \frac{p(z)}{z^2} \ {\rm d}z,
\]
holds for a.e. $\tau \in (0,T)$.
\end{itemize}
\end{Definition}

The existence of dissipative solutions can be shown by the method of
Lions \cite{LI4}, with the necessary modifications introduced in
\cite{FNP}. 

\begin{Remark} \label{RR1}
In the Navier-Stokes limit, we will impose an extra boundary condition:
\begin{equation} \label{slip}
\vu (x_h, z) = 0 \qquad \forall (x_h,z) \in \Ov{\Omega}_\ep,\quad z \in \{ 0,1 \}.
\end{equation}
Accordingly, the class of admissible test functions in the momentum balance (\ref{mb}) is restricted to:
\[
\varphi \in C^\infty ([0,T] \times \Ov{\Omega}_\ep; \mathbb{R}^3),\qquad
\varphi \cdot \vc{n}|_{\partial \Ome} = 0, \qquad \varphi \ \mbox{compactly supported in}\ z \in (0,1).
\]
\end{Remark}

\subsection{Main results}

Our goal is to identify the asymptotic limit for solutions of system
(\ref{i3}--\ref{i5}), (\ref{i7})/(\ref{slip}) if the diameter $\ep$ of
the cylinder $\Ome$ tends to zero. To measure the distance to the
solutions of the limit system, we use the relative energy functional: 
\begin{equation}
\label{r1} {\mathcal E}_\ep \left( \vr, \vu \ \Big| r, \vc{U} \right) =
\intOe{ \left( \frac{1}{2} \vr |\vu - \vc{U}|^2 + H(\vr) - H'(r)
(\vr - r) - H(r) \right)}. 
\end{equation}
{Since $H''(\vr) = p'(\vr)/\vr$ and} the pressure $p$ is a
{strictly} increasing differentiable function of the density, the
pressure potential $H$ is strictly convex  and it is easy to
check that {for $r > 0$}: 
\[
{\mathcal E}_\ep \left( \vr, \vu \ \Big| r, \vc{U} \right) =0 \ \Leftrightarrow  \ \vr = r, \ \vu = \vc{U}.
\]
Moreover, it follows from (\ref{press}) that:
\begin{equation} \label{coerc}
\begin{split}
&C_1(K) \left( \left| \vu - \vc{U} \right|^2 + \left| \vr - r \right|^2 \right) \leq
\frac{1}{2} \vr |\vu - \vc{U}|^2 + H(\vr) - H'(r) (\vr - r) - H(r) \\
& \qquad\qquad\qquad\qquad \leq C_2(K)\left( \left| \vu - \vc{U} \right|^2 + \left| \vr - r \right|^2 \right)
\qquad \forall  \vr, r \in K \subset (0, \infty), \ K \ \mbox{compact}\\
&\mbox{and }\quad \frac{1}{2} \vr |\vu - \vc{U}|^2 + H(\vr) - H'(r) (\vr - r) - H(r) \geq C(K, \tilde K)
\left( 1 + \vr |\vu - \vc{U}|^2 + \vr^\gamma \right) \\
& \qquad\qquad\qquad \qquad \forall r \in K \subset {\rm int}[\tilde K],\
\vr \in [0, \infty) \setminus \tilde K, \ \tilde K \subset (0, \infty) \ \mbox{compact}.
\end{split}
\end{equation}

\subsubsection{Inviscid limit}

The system (\ref{i1}), (\ref{i2}) can be written as a semilinear perturbation of the standard isentropic Euler
system {in the following form}:
\begin{align*}
\partial_t \vr_E + \partial_z (\vr_E u_E) + \frac{\partial_z A}{A}
\vr_E u_E &= 0,
\\
\partial_t (\vr_E u_E) + \partial_z (\vr_E u_E^2) + \partial_z p(\vr_E) + \frac{\partial_z A}{A} \vr_E u_E^2 &= 0.
\end{align*}
In view of the standard theory of hyperbolic conservation laws, see e.g. Majda \cite{Majd}, one can therefore
anticipate the existence of \emph{local} in time smooth solutions to
problem (\ref{i1}), (\ref{i2}) provided the initial data are smooth
enough. As shown in the following theorem, these solutions may be seen
as suitable limits of those of the Navier-Stokes system
(\ref{i3}--\ref{i5}), (\ref{i7}) in $\Ome$ in the regime $\ep, \lambda\to 0$. 

\Cbox{Cgrey}{
\begin{Theorem} \label{Tinv}
Let $\Omega_\ep$ be given by (\ref{i6}), where $\Omega = \Omega_1$ is determined through (\ref{hyp0}), with $\vc{V}_h \in C^1(\Ov{\Omega}; \mathbb{R}^2)$.
Let the pressure $p$ satisfy hypothesis (\ref{press}).
Set:
\[
A(z) = |\omega(z)|.
\]
Let $[\vr_E, u_E]$ be a classical solution of the Euler system (\ref{i1}), (\ref{i2}) on a time
interval $[0,T]$ such that:
\begin{equation} \label{sleu}
u_E|_{ z \in \{ 0,1 \} } = 0.
\end{equation}
Let $[\vr, \vu]$ be a (weak) dissipative solution of the Navier-Stokes system (\ref{i3}--\ref{i5}), (\ref{i7})
in $(0,T) \times \Ome$.

Then there is a constant $C$, depending only on time $T$, 
on the norm of the solution $[\vr_E, u_E]$, on
the $C^1$ norm of $\vc{V}_h$, but independent of
$[\vr, \vu]$ and of the scaling parameters $\lambda$ and $\ep$, such that:
\begin{equation} \label{euler}
\frac{1}{|\Ome|} \mathcal{E}_\ep \left( \vr, \vu \ \Big| \ \vr_E, \vu_E \right) (\tau) \leq C \left(
\lambda + \ep + \frac{1}{|\Ome|} \mathcal{E}_\ep \left( \vr, \vu \ \Big| \ \vr_E, \vu_E \right) (0) \right)
\end{equation}
for any $\tau \in (0,T)$, where we have set $\vu_E = [0,0,u_E]$.
\end{Theorem}
}
Theorem \ref{Tinv} will be shown in Section \ref{CE}.

\subsubsection{Positive viscosity limit}

Similarly to the preceding section, we may rewrite \eqref{i1a}, \eqref{i2a} as:
\begin{align*}
\partial_t \vr_{NS} + \partial_z (\vr_{NS} u_{NS}) + \frac{\partial_z
  A}{A} \vr_{NS} u_{NS} &= 0,
\\
\partial_t (\vr_{NS} u_{NS}) + \partial_z (\vr_{NS} u_{NS}^2)
+ \partial_z p(\vr_{NS}) + \frac{\partial_z A}{A} \vr_{NS} u_{NS}^2 
&= \left( \frac{4 \mu}{3} + \eta \right)  \partial_z^2 u_{NS} + 
\left( \frac{\mu}{3} + \eta\right) \partial_z \left(\frac{\partial_z A}{A} u_{NS}\right). 
\end{align*}
Thus, by analogy to its inviscid counterpart, we may anticipate the existence of at least local-in-time smooth
solutions to system \eqref{i1a}, \eqref{i2a}, supplemented with the boundary conditions:
\begin{equation} \nonumber 
u_{NS}|_{ z\in\{ 0,1 \} }  = 0,
\end{equation}
for sufficiently smooth initial data. Moreover, in view of the theory developed by Kazhikhov \cite{KHA2},
we may even expect those solutions to be global in time, however, we were not able to find a relevant reference.
We claim the following result proved in Section \ref{CNS}.

\Cbox{Cgrey}{
\begin{Theorem} \label{Tvisc}
Let $\Omega_\ep$ be given by (\ref{i6}), where $\Omega$ is determined
through (\ref{hyp0}), with the vector field $\vc{V}_h \in
C^2(\Ov{\Omega}; \mathbb{R}^2)$. Let the pressure $p$ satisfy hypothesis (\ref{press}).
Set:
\[
A(z) = |\omega(z)|.
\]
Let $[\vr_{NS}, u_{NS}]$ be a classical solution of the Navier-Stokes
system with drift (\ref{i1a}), (\ref{i2a}) on a time interval $[0,T]$, satisfying:
\[
u_{NS}|_z\in\{ 0,1 \} = 0.
\]
Let $[\vr, \vu]$ be a (weak) dissipative solution of the Navier-Stokes system (\ref{i3}--\ref{i5}), (\ref{i7})
in $(0,T) \times \Ome$ with $\lambda = 1$ and strictly positive
  bulk viscosity $\eta > 0$, satisfying, in addition, the no-slip
boundary condition (\ref{slip}) at the horizontal part of the boundary of the cylinder $\Ome$.

Then there is a constant $C$, depending only on time $T$, on
the norm of the solution $[\vr_{NS}, u_{NS}]$, on the 
$C^2$ norm of the vector field $\vc{V}_h$, but independent of
$[\vr, \vu]$ and of the scaling parameter $\ep$, such that:
\begin{equation} \nonumber
\frac{1}{|\Ome|} \mathcal{E}_\ep \left( \vr, \vu \ \Big| \ \vr_{NS}, \vu_{NS} \right) (\tau) \leq C \left(
\ep + \frac{1}{|\Ome|} \mathcal{E}_\ep \left( \vr, \vu \ \Big| \ \vr_{NS}, \vu_{NS} \right) (0) \right)
\end{equation}
for any $\tau \in (0,T)$, where we have set $\vu_{NS} = [0,0,u_{NS}]$.
\end{Theorem}
}
As already pointed out, the proof of Theorem \ref{Tvisc} is based on a
version of Korn-Poincar\' e inequality on thin domains proved in
Section \ref{K}. 

\section{The relative energy inequality}\label{r}

As shown in \cite{FeJiNo}, \emph{any} dissipative solution $[\vr,\vu]$ of the Navier-Stokes system
(\ref{i3}--\ref{i5}) satisfies the \emph{relative energy inequality}:
\begin{equation}
\label{r3}
\begin{split}
{\mathcal E}_\ep \left( \vr, \vu \ \Big| r ,
\vc{U} \right)(\tau) &+ \int_0^\tau \intOe{ \Big( \tn{S} (\Grad
\vu) - \tn{S} (\Grad \vc{U}) \Big): \Big( \Grad \vu - \Grad \vc{U}
\Big) } \ \dt
\\
&\leq \ {\mathcal E}_\ep \left( \vr(0, \cdot), \vu(0, \cdot) \ \Big| \ r(0, \cdot), \vc{U}(0, \cdot) \right) +
\int_0^\tau {\mathcal R}_\ep (\vr, \vu, r, \vc{U} ) \ \dt,
\end{split}
\end{equation}
with the remainder:
\begin{equation}\nonumber
\begin{split}
{\mathcal R}_\ep \left( \vr, \vu, r, \vc{U} \right) = &
\intOe{  \vr \Big( \partial_t \vc{U} + \vu \Grad \vc{U} \Big) \cdot
(\vc{U} - \vu )}
+ \lambda\intOe{\tn{S}(\Grad \vc{U}):\Grad (\vc U- \vc{u})  }
\\ & + \intOe{ \left( (r - \vr) \partial_t H'(r) + \Grad H'(r) \cdot
\left( r \vc{U} - \vr \vu \right) \right) }
\\ & - \intOe{ \Div \vc{U} \Big( p(\vr) - p(r) \Big) }.
\end{split}
\end{equation}
Here $[r,\vc{U}]$ represent arbitrary test functions that are sufficiently smooth and satisfy a kind of
compatibility conditions:
\begin{equation} \label{r4a}
r > 0, \qquad \vc{U} \cdot \vc{n}|_{\partial \Ome} = 0,
\end{equation}
and:
\begin{equation} \label{r4b}
\vc{U} (x_h, z) = 0 \qquad\forall (x_h,z) \in \Ov{\Omega}_\ep,\ { z\in\{ 0,1 \} }
\end{equation}
provided the extra no-slip condition (\ref{slip}) is imposed.

\subsection{Extending the velocity field}

The proofs of Theorems \ref{Tinv}, \ref{Tvisc} are based on the idea to
use the solutions of the target systems {to construct} test
functions for the relative energy inequality (\ref{r3}). This
cannot be done directly as the velocity fields $u_E$, $u_{NS}$ or, more specifically, their extensions $\vu_E 
= [0,0,u_E]$, $\vu_{NS} = [0,0,u_{NS}]$ do not comply with the
boundary conditions (\ref{r4a}), (\ref{r4b}), respectively. Instead,
we consider a {tilted} extension of a velocity field of the form:
\begin{equation} \label{g2}
\vc{U}_\ep = \left[ \vc{V}_{h, \ep} , 1 \right] v,\quad v = u_E, u_{NS},
\end{equation} 
where:
\bFormula{g2A}
\vc{V}_{h,\ep}(\xh,z) := \ep \vc{V}_h \left( \frac{\xh}{\ep}, z \right) \qquad \forall(\xh,z) \in \Ov{\Omega}_\ep
\eF
and $\vc{V}_h$ is the vector field introduced in (\ref{hyp0}).
As the vector field $ \left[ \vc{V}_{h,\ep}, 1 \right]$
is tangent to $\partial \Ome$ at any point of the lateral boundary $\partial \Ome \cap \{ 0 < z < 1 \}$,
$\vc{U}_\ep$ is an admissible test function in (\ref{r3}) as soon as
$u_E$, $u_{NS}$ vanish at $z\in\{ 0,1 \}$. The following
result shows that the extension defined through (\ref{g2}) satisfies
also the equation of continuity. 

\begin{Lemma} \label{Lg1}
Let $\vc{U}_\ep$ be the velocity field defined by (\ref{g2}) and
suppose that the functions $r = r(z)$, $v = v(z)$ satisfy: 
\[
\partial_t \left( r A \right) + \Div \left(r v A \right) =  \partial_t \left( r A \right) + \partial_z \left(r v A \right) = 0
\qquad \forall z \in (0,1),
\]
where $A(z) = |\omega_h(z)|$. Then:
\[
\partial_t r + \Div (r \vc{U}_\ep ) = 0 \qquad  \mbox{in } {\Ome}.
\]
\end{Lemma}
\begin{proof}
On one hand, we have:
\[
\partial_t (r A) + \Div (r \vc{U}_\ep A) = A \left( \partial_t r + \Div (r \vc{U}_\ep) \right) + r v \partial_z A.
\]
On the other hand, in accordance with (\ref{g1a}), we get:
\[
\begin{split}
\partial_t (r A) + \Div (r \vc{U}_\ep A) & =  \partial_t (rA) + \partial_z (r vA ) + {\rm div}_h (r \vc{V}_{h, \ep} v A)
= rvA {\rm div}_h \vc{V}_{h, \ep} =  rvA {\rm div}_h \vc{V}_{h} = rv \partial_z A,
\end{split}
\]
and the desired conclusion follows.
\end{proof}

\subsection{Relative energy inequality and the asymptotic limits}

We start by rewriting $\mathcal{R}_\ep$ as:
\begin{equation}
\label{r5}
\begin{split}
{\mathcal R}_{\ep}\left( \vr, \vu, r, \vc{U} \right) = &
\intOe{  \vr \Big( \partial_t \vc{U} + \vc{U} \cdot \Grad \vc{U} \Big) \cdot(\vc{U} - \vu) } - \intOe{ \vr (\vu - \vc{U}) \cdot \Grad \vc{U} \cdot (\vc{U} - \vu) }
\\ & + \lambda\intOe{\tn{S}(\Grad \vc{U}):\Grad (\vc U- \vc{u})  }
\\ & + \intOe{ \left( (r - \vr) \partial_t H'(r) + \Grad H'(r) \cdot
\left( r \vc{U} - \vr \vu \right) \right) }
\\ & - \intOe{ \Div \vc{U} \Big( p(\vr) - p(r) \Big) }.
\end{split}
\end{equation}

\subsubsection{Relative energy inequality in the inviscid limit}
\label{EU}

Take $r = \vr_E$, $\vc{U} = \vc{U}_\ep = [ \vc{V}_{h, \ep}, 1] u_E$ as test functions in the relative energy inequality
(\ref{r3}), where $\vr_E = \vr_E(z)$, $u_E = u_E(z)$ is a (smooth) solution of the 1D-Euler system (\ref{i1}), (\ref{i2})
satisfying the boundary conditions (\ref{sleu}).
Going back to (\ref{r5}) we compute:
\begin{equation} \label{g3}
\begin{split}
\intOe{  \vr \Big( \partial_t \vc{U}_{\ep} + \vc{U}_{\ep} \cdot &\Grad \vc{U}_{\ep} \Big) \cdot(\vc{U}_{\ep} - \vu) }
\\ & = \intOe{  \vr \Big( \partial_t \vc{U}_{\ep} - \partial_t \vu_E + \vc{U}_{\ep} \cdot \Grad \vc{U}_{\ep} -
\vu_E \cdot \Grad \vu_E \Big) \cdot(\vc{U}_{\ep} - \vu) }\\
&\quad + \intOe{  \vr \Big( \partial_t u_E + u_E \cdot \partial_z u_E \Big) (u_E - u_3) }\\
&= E_1 (\vr, \vc{U}_\ep, u_E, \vu) - \intOe{ \frac{\vr}{\vr_E } \partial_z p(\vr_E) (u_E - u_3)},
\end{split}
\end{equation}
where {the last equality follows from $\vr_E(\partial_t u_E + u_E
  \cdot \partial_z u_E + \partial_z p(\vr_E)) = 0$, which is a
  consequence of \eqref{i1} and \eqref{i2}, and the error term has the form}:
\[
E_1 (\vr, \vc{U}_\ep, u_E, \vu) =
\intOe{  \vr \Big( \partial_t \vc{U}_{\ep} - \partial_t \vu_E + \vc{U}_{\ep} \cdot \Grad \vc{U}_{\ep} -
\vu_E \cdot \Grad \vu_E \Big) \cdot(\vc{U}_{\ep} - \vu) }.
\]
Next, the terms containing the pressure, coming from the last line in \eqref{r5}, can be written as:
\[
\begin{split}
&\intOe{ \left( (\vr_E - \vr) \frac{1}{\vr_E} p'(\vr_E) \partial_t \vr_E + \frac{1}{\vr_E} p'(\vr_E) \partial_z \vr_E
\left( \vr_E u_E - \vr u_3 \right) \right) }\\
&= \intOe{ \partial_t p(\vr_E) + \partial_z p(\vr_E) u_E } - \intOe{ \frac{\vr}{\vr_E} p'(\vr_E) \left( \partial_t \vr_E +
\partial_z \vr_E u_3 \right)
}\\
&= \intOe{ \partial_t p(\vr_E) + \partial_z p(\vr_E) u_E } - \intOe{ \frac{\vr}{\vr_E} p'(\vr_E) \left( \partial_t \vr_E +
\partial_z \vr_E u_E \right)} + \intOe{ \frac{\vr}{\vr_E } \partial_z p(\vr_E) (u_E - u_3) }.
\end{split}
\]

Finally, we use the fact established in Lemma \ref{Lg1}, namely that
$[\vr_E, \vc{U}_\ep]$ solve the equation of continuity, to conclude: 
\begin{equation} \label{g4}
\begin{split}
&\intOe{ \left( (\vr_E - \vr) \frac{1}{\vr_E} p'(\vr_E) \partial_t \vr_E + \frac{1}{\vr_E} p'(\vr_E) \partial_z \vr_E
\left( \vr_E u_E - \vr u_3 \right) \right) }\\
&\qquad\qquad 
=  \intOe{ p'(\vr_E) \Big( \vr - \vr_E \Big) \Div \vc{U}_\ep } + \intOe{ \frac{\vr}{\vr_E } \partial_z p(\vr_E) (u_E - u_3) }.
\end{split}
\end{equation}
Thus, summing up (\ref{g3}), (\ref{g4}) and comparing the resulting expression with (\ref{r5}), we may infer that:
\begin{equation}
\label{g5}
\begin{split}
{\mathcal R}_\ep\left( \vr, \vu, \vr_E, \vc{U}_\ep \right) = &
\intOe{
\Div \vc{U}_\ep \Big( p(\vr) - p'(\vr_E) (\vr - \vr_E) -  p(\vr_E) \Big)}\\
&- \intOe{ \vr (\vu - \vc{U}_\ep) \cdot \Grad \vc{U}_\ep \cdot (\vc{U}_\ep - \vu) }
\\
&+ \lambda \intOe{\tn{S}(\Grad \vc{U}_\ep ):\Grad (\vc{U}_\ep - \vc{u})  } + E_1 (\vr, \vc{U}_\ep, u_E, \vu).
\end{split}
\end{equation}

\subsubsection{Relative entropy inequality in the viscous limit}
\label{NS}

The viscous (Navier-Stokes) limit can be handled in a similar way. An
analogue of (\ref{g3}), derived using \eqref{i1a} and \eqref{i2a}, reads:
\begin{equation} \nonumber 
\begin{split}
&\intOe{  \vr \Big( \partial_t \vc{U}_{\ep} + \vc{U}_{\ep} \cdot \Grad \vc{U}_{\ep} \Big) \cdot(\vc{U}_{\ep} - \vu) }\\
&\qquad= E_1 (\vr, \vc{U}_\ep, u_{NS}, \vu) - \intOe{ \frac{\vr}{\vr_{NS} } \partial_z p(\vr_{NS}) (u_{NS} - u_3) }
\\ &\qquad \quad
+ \intOe{ \frac{\vr}{\vr_{NS}} \left( \nu \partial^2_z u_{NS} + (\mu/3 + \eta)\partial_z ( \partial_z (\ln A) u_{NS} ) \right) (u_{NS} - u_3) },
\end{split}
\end{equation}
which, after a similar treatment as in Section \ref{EU} gives rise to the remainder:
\begin{equation}
\label{ns2}
\begin{split}
{\mathcal R}_\ep( \vr, \vu, &\vr_{NS}, \vc{U}_\ep)  = 
\int_{\Omega_\ep}  \Div \vc{U}_\ep \Big( p(\vr) - p'(\vr_{NS}) (\vr - \vr_{NS}) -  p(\vr_{NS}) \Big)~\mbox{d}x\\
& \qquad\qquad\quad - \intOe{ \vr (\vu - \vc{U}_\ep) \cdot \Grad \vc{U}_\ep \cdot (\vc{U}_\ep - \vu) }\\
& \qquad\qquad\quad + \intOe{ \frac{\vr}{\vr_{NS}} \left( \nu \partial^2_z u_{NS}
    + (\mu/3 + \eta)\partial_z ( \partial_z (\ln A) u_{NS} )
  \right) (u_{NS} - u_3) } 
\\ & \qquad\qquad\quad - \intOe{ \Div \tn{S}(\Grad \vc{U}_\ep ) \cdot (\vc{U}_\ep - \vc{u})  }
+ E_1 (\vr, \vc{U}_\ep, u_{NS}, \vu)\\
& = \intOe{ \Div \vc{U}_\ep \Big( p(\vr) - p'(\vr_{NS}) (\vr - \vr_{NS}) -  p(\vr_{NS}) \Big)}\\
& \quad - \intOe{ \vr (\vu - \vc{U}_\ep) \cdot \Grad \vc{U}_\ep \cdot (\vc{U}_\ep - \vu) }\\
& \quad + \intOe{ \frac{1}{\vr_{NS}} \Big( \vr - \vr_{NS}
  \Big) \left( \nu \partial^2_z u_{NS} + (\mu/3 +
      \eta)\partial_z ( \partial_z (\ln A) u_{NS} ) \right) \Big(u_{NS} - u_3 \Big) }\\ 
& \quad +  E_1 (\vr, \vc{U}_\ep, u_{NS}, \vu) + E_2(\vc{U}_\ep, u_{NS}, \vu),
\end{split}
\end{equation}
where we have set:
\[
\begin{split}
E_2(\vc{U}_\ep, u_{NS}, \vu) = & \intOe{ \left( \nu \partial^2_z u_{NS}
    + (\mu/3 + \eta)\partial_z ( \partial_z (\ln A) u_{NS} )
  \right) \Big( u_{NS} - u_3 \Big) } \\ & -\intOe{ \Div \tn{S}(\Grad
  \vc{U}_\ep ) \cdot (\vc{U}_\ep - \vc{u})  }. 
\end{split}
\]

\subsection{Estimates of the error terms}

Our goal is to show that the error terms $E_1$, $E_2$ vanish in the asymptotic limit $\ep \to 0$.
As for $E_1$, we first observe that:
\begin{equation}
\nonumber
\sup_{x \in \Ome} \left| \vc{U}_\ep - \vu_E \right| = \sup_{x \in \Ome} \left| u_E \vc{V}_{h, \ep} \right| \leq C\ep.
\end{equation}
Moreover, seeing that:
\[
\partial_t \vc{U}_\ep - \partial_t \vu_E = \partial_t u_E \vc{V}_{h,\ep},
\]
we deduce:
\[
\left\| \partial_t \vc{U}_\ep - \partial_t \vu_E \right\|_{C([0,T] \times \Ov{\Omega}_\ep)} +
\left\| \vc{U}_\ep \cdot \Grad \vc{U}_\ep - u_E \partial_z \vc{U}_{\ep} \right\|_{C([0,T] \times \Ov{\Omega}_\ep)} \leq C\ep.
\]

Finally, we estimate:
\[
\left\| u_E \partial_z \vc{U}_\ep - \vu_E \cdot \nabla_x \vu_E \right\|_{C([0,T] \times \Ov{\Omega}_\ep)} =
\left\| u_E \partial_z \left( u_E \vc{V}_{h,\ep} \right) \right\|_{C([0,T] \times \Ov{\Omega}_\ep)} \leq C\ep,
\]
obtaining:
\begin{equation} \label{ns3}
\left| E_1 (\vr, \vc{U}_\ep, u_E, \vu) \right| \leq C\ep \intOe{ \vr |\vc{U}_\ep - \vu |},
\end{equation}
{provided that $u_E$ is continuously differentiable in $[0,T] \times [0,1]$.}
Similarly, we can show that:
\begin{equation} \label{ns4}
\left| E_1 (\vr, \vc{U}_\ep, u_{NS}, \vu) \right| \leq C\ep \intOe{ \vr |\vc{U}_\ep - \vu | }
\end{equation}
provided $u_{NS}$ is continuously differentiable in $[0,T] \times [0,1]$.

To control $E_2$, we use:
$$\Div \tn{S}(\Grad \vc{U}_\ep) = \mu \Delta \vc{U}_\ep + \left(
  \frac{\mu}{3} + \eta\right) \nabla_x \Div \vc{U}_\ep\quad \mbox{and}
\quad \vc{U}_\ep(\xh,z) = [\vc{V}_{h,\ep}(\xh,z),1]u_{NS}(z)$$ 
and we write:
\begin{align*}
  \intOe{ \Div \tn{S}(\Grad \vc{U}_\ep ) \cdot (\vc{U}_\ep - \vc{u}) }
&= \intOe{ \mu \Delta \vc{U}_\ep \cdot (\vc{U}_\ep - \vu) + \left( \frac{\mu}{3} + \eta\right) \Grad \Div \vc{U}_\ep \cdot (\vc{U}_\ep - \vu) }
\\
&= \intOe{ \mu [ \Delta_h (\vc{V}_{h,\ep}) u_{NS} + \partial^2_z (\vc{V}_{h,\ep} u_{NS}) ]\cdot (\vc{U}_\ep - \vu)_h }
\\
&\quad + \intOe{ \mu \partial^2_zu_{NS}(u_{NS}-u_3)}
\\
&\quad + \intOe{ \left(\frac{\mu}{3} + \eta \right) [ \nabla_h \Divh (\vc{V}_{h,\ep}) u_{NS} + \nabla_h \partial_{z} u_{NS} ]\cdot (\vc{U}_\ep - \vu)_h }
\\
&\quad + \intOe{ \left(\frac{\mu}{3} + \eta \right) [ \partial_z ( \Divh (\vc{V}_{h,\ep}) u_{NS} ) + \partial^2_z u_{NS} ] (u_{NS} - u_3) }.
 \end{align*}
Since $\vc{V}_{h,\ep}(\xh,z)$ is given by (\ref{g2A}) with $\vc{V}_h$ satisfying (\ref{hyp0}), by assumptions of Theorem~\ref{Tvisc} the first and second derivative 
of $\vc{V}_{h,\ep}$ in the $z$-variable are bounded by
$C\ep$. Moreover: $\Delta_h \vc{V}_{h,\ep} = 0$, $\nabla_h \Divh(\vc{V}_{h,\ep})=0$, 
and  $|\partial^2_z(\vc{V}_{h,\ep}u_{NS})|\le C\ep$ provided that
$\partial^2_zu_{NS}$ is bounded in $[0,T]\times [0,1]$.
Since $u_{NS}$ is a function of $z$ only, we also see that
$\nabla_h \partial_z u_{NS} = 0$. Using $\Divh \vc{V}_{h,\ep} = \Divh
\vc{V}_h = \partial_z(\ln(A))$ in view of (\ref{g1a}), the above implies: 
\begin{multline}
\biggl| \intOe{ \Div \tn{S}(\Grad \vc{U}_\ep ) \cdot (\vc{U}_\ep - \vc{u}) }
- \left(\frac{4}{3}\mu + \eta \right)
\intOe{  \partial^2_zu_{NS}(u_{NS}-u_3)} \\
- \left(\frac{\mu}{3} + \eta \right)\intOe{  \partial_z
  ( \partial_z(\ln(A) u_{NS} ) (u_{NS} - u_3) }\biggr| \le C\ep \intOe
{|\vu - \vc{U}_\ep|}. 
\end{multline}
Consequently, we get:
\begin{equation}
\label{ns5}
\left| E_2(\vc{U}_\ep, u_{NS}, \vu) \right| \leq C \ep \intOe{ |\vu - \vc{U}_\ep | }
\end{equation}
provided that $\partial^2_{z} u_{NS}$ is bounded in $[0,T] \times [0,1]$.

\section{Convergence} \label{c}

Having collected the necessary material, we are now ready to complete the proofs of Theorems \ref{Tinv}, \ref{Tvisc}.
As the solutions of the limit systems are regular, we may assume: 
\[
0 < \underline{\vr} \leq \vr_E \leq \Ov{\vr},\qquad 0 < \underline{\vr} \leq \vr_{NS} \leq \Ov{\vr}
\]
for certain positive constants $\underline{\vr}$, $\Ov{\vr}$.
Next, it is convenient to introduce
the \emph{essential} and \emph{residual} component of an integrable function $h$ as:
\[
h_{\rm ess} = \chi( \vr ) h, \qquad h_{\rm res} = (1 - \chi(\vr)) h,
\]
where:
\[
\chi \in \DC(0, \infty), \qquad 0 \leq \chi \leq 1, \qquad \chi(z) = 1
\qquad \forall z \in [\underline{\vr}/2, 2 \Ov{\vr} ].
\]

\subsection{Convergence to the Euler system - the proof of Theorem \ref{Tinv}}
\label{CE}

It follows from the relative energy inequality (\ref{r3}), the
coercivity {\eqref{coerc}}, and the bounds (\ref{g5}), (\ref{ns3}) that: 
\[
\begin{split}
\left[ \mathcal{E}_\ep \left( \vr, \vu \ \Big|\ \vr_E, \vc{U}_\ep \right)(t) \right]_{t=0}^{t = \tau}
&+ \lambda \int_0^\tau \intOe{ \left( \tn{S}(\Grad \vu) - \tn{S}(\Grad
    \vc{U}_\ep ) \right) : \left( \Grad \vu - \Grad \vc{U}_\ep \right)}  \ \dt\\ 
&\leq C \int_0^\tau \mathcal{E}_\ep \left( \vr, \vu \ \Big|\ \vr_E,
  \vc{U}_\ep \right)(t) \ \dt + C \ep \int_0^\tau \intOe{ \vr |\vu -
  \vc{U}_\ep | } \\ 
& \quad + \lambda \int_0^\tau \intOe{ \tn{S} (\Grad \vc{U}_\ep ) : \Grad (\vc{U}_\ep - \vu ) } \ \dt,
\end{split}
\]
where, furthermore:
\[
\begin{split}
\ep \int_0^\tau \intOe{ \vr |\vu - \vc{U}_\ep | }  & \leq \frac{\ep}{2}
\int_0^\tau \intOe{ \vr |\vu - \vc{U}_\ep |^2 } +  \frac{\ep}{2}
\int_0^\tau \intOe{ \vr } \\ & \leq C \left( \ep |\Ome| + \int_0^\tau \mathcal{E}
  \left( \vr, \vu \ \Big|\ \vr_E, \vc{U}_\ep \right)(t) \ \dt \right). 
\end{split}
\]

Next, setting $\tilde \vu := \vc{U}_\ep - \vu$ for notational convenience, we write:
\begin{align*}
 \tn{S} (\Grad \vc{U}_\ep ) : \Grad (\vc{U}_\ep - \vu ) &= \mu( \Grad
 \vc{U}_\ep + \Grad^t \vc{U}_\ep - \frac{2}{3}\Div \vc{U}_\ep \tn{I})
 : \Grad \tilde u + \eta \Div \vc{U}_\ep \Div \tilde u  \\
&= \frac{\mu}{2} ( \Grad \vc{U}_\ep + \Grad^t \vc{U}_\ep -
\frac{2}{3}\Div \vc{U}_\ep \tn{I}) : ( \Grad \tilde \vu + \Grad^t
\tilde \vu - \frac{2}{3} \Div \tilde \vu) + \eta \Div \vc{U}_\ep \Div
\tilde u, 
\end{align*}
where we used the fact that $\Grad \vc{U}_\ep + \Grad^t \vc{U}_\ep -
\frac{2}{3}\Div \vc{U}_\ep \tn{I}$ is symmetric and traceless to
smuggle in $\Grad^t \tilde \vu$ and $\frac{2}{3} \Div \tilde \vu$. In
a similar way, we observe that:
\begin{align*}
 \left( \tn{S}(\Grad \vu) - \tn{S}(\Grad \vc{U}_\ep ) \right) : \left(
   \Grad \vu - \Grad \vc{U}_\ep \right) = \tn{S}(\Grad \tilde \vu) :
 \Grad \tilde \vu = \frac{\mu}{2} \left| \Grad \tilde \vu + \Grad^t
   \tilde \vu - \frac{2}{3} \Div \tilde \vu \tn{I} \right|^2 + \eta
 |\Div \tilde \vu|^2, 
\end{align*}
and so, using the above, we may estimate:
\begin{align*}
& \left| \lambda \int_0^\tau \intOe{ \tn{S} (\Grad \vc{U}_\ep ) : \Grad
    (\vc{U}_\ep - \vu ) } \ \dt \right| \\
& \qquad\qquad \leq C\lambda \int_0^\tau
\intOe{ | \Grad \vc{U}_\ep |^2 }  + \frac{\lambda}{2} \int_0^\tau \intOe{
  \frac{\mu}{2} \left| \Grad \tilde \vu + \Grad^t \tilde \vu -
    \frac{2}{3} \Div \tilde \vu \tn{I} \right|^2 + \eta |\Div \tilde \vu|^2 }\\ 
& \qquad\qquad \leq C \lambda | \Ome | + \frac{\lambda}{2} \int_0^\tau \intOe{
  \left( \tn{S}(\Grad \vu) - \tn{S}(\Grad \vc{U}_\ep ) \right) :
  \left( \Grad \vu - \Grad \vc{U}_\ep \right) }. 
\end{align*}

Combining the previous estimates and a Gronwall-type argument, we conclude:
\begin{equation} \nonumber 
\mathcal{E}_\ep \left( \vr, \vu \ \Big|\ \vr_E, \vc{U}_\ep \right) (\tau) \leq C \left[ \left(\ep + \lambda \right) |\Ome| +
\mathcal{E}_\ep \left( \vr, \vu \ \Big|\ \vr_E, \vc{U}_\ep \right) (0)
\right] \qquad \forall 0 \leq \tau \leq T,
\end{equation}
from which we easily deduce (\ref{euler}). We have proved Theorem \ref{Tinv}. \qed

\subsection{Convergence to the Navier-Stokes system - the proof of
  Theorem \ref{Tvisc}} \label{CNS}

Proving similar estimates for the Navier-Stokes limit is more delicate. We start observing that (\ref{ns2}), \eqref{ns4} together with \eqref{ns5} and the coercivity property \eqref{coerc}, give rise to: 
\begin{equation} \label{C1}
\begin{split}
&\left[ \mathcal{E}_\ep \left( \vr, \vu \ \Big|\ \vr_{NS}, \vc{U}_\ep \right)(t) \right]_{t=0}^{t = \tau}
+ \int_0^\tau \intOe{ \left( \tn{S}(\Grad \vu) - \tn{S}(\Grad
    \vc{U}_\ep ) \right) : \left( \Grad \vu - \Grad \vc{U}_\ep \right)
}  \ \dt\\ 
& \leq C \int_0^\tau \mathcal{E}_\ep \left( \vr, \vu \ \Big|\ \vr_{NS},
  \vc{U}_\ep \right)(t) \ \dt + C \ep \int_0^\tau \intOe{ \vr |\vu -
  \vc{U}_\ep | } + C \ep \int_0^\tau \intOe{ |\vu - \vc{U}_\ep | } \\ 
&  + \int_0^\tau \intOe{ \frac{1}{\vr_{NS}} \Big( \vr - \vr_{NS} \Big)
  \left( \left( {4 \mu}/{3} + \eta \right) \partial^2_z u_{NS} +
    (\mu/3 + \eta)\partial_z ( \partial_z (\ln A) u_{NS} )
  \right) \Big( u_{NS} - u_3 \Big) } \dt, 
\end{split}
\end{equation}
where the integral in the last line, using the notation:
\[
F(z) := \left( (4 \mu/3 + \eta) \partial^2_z u_{NS} +(\mu/3 + \eta)\partial_z ( \partial_z (\ln A) u_{NS} ) \right),
\]
the fact that $|F| \le C$ and \eqref{coerc}, can be estimated by the following:
\[
\begin{split}
&\left| \intOe{ \frac{F}{\vr_{NS}} \Big( \vr - \vr_{NS} \Big) \Big( u_{NS} - u_3 \Big) } \right|\\
&\qquad\qquad \leq \left| \intOe{ \frac{F}{\vr_{NS}} \Big[ \vr
    - \vr_{NS} \Big]_{\rm ess}  \Big[ u_{NS} - u_3 \Big]_{\rm ess} }
\right| + \left| \intOe{ \frac{F}{\vr_{NS}} \Big[ \vr -
    \vr_{NS} \Big]_{\rm res} \Big[ u_{NS} - u_3 \Big]_{\rm res} }
\right|\\ 
& \qquad\qquad \leq C \left[ \mathcal{E}\left( \vr, \vu \ \Big|\ \vr_{NS}, \vc{U}_\ep \right) + C(\delta) \intOe{ \left|
[\vr - \vr_{NS} ]_{\rm res} \right| } \right] + 
\delta \intOe{ (1 + \vr) \left| [ \vu - \vc{U}_\ep ]_{\rm res} \right|^2 }
\\ & \qquad\qquad \leq C(\delta) \mathcal{E}\left( \vr, \vu \ \Big|\ \vr_{NS}, \vc{U}_\ep \right) +
\delta \intOe{ |\vu - \vc{U}_\ep |^2},
\end{split}
\]
for any $\delta > 0$.
Applying a similar treatment to the remaining integrals in (\ref{C1}), we obtain that:
\[
\begin{split}
&\left[ \mathcal{E}_\ep \left( \vr, \vu \ \Big|\ \vr_{NS}, \vc{U}_\ep \right)(t) \right]_{t=0}^{t = \tau}
+ \int_0^\tau \intOe{ \left( \tn{S}(\Grad \vu) - \tn{S}(\Grad \vc{U}_\ep ) \right) : \left( \Grad \vu - \Grad \vc{U}_\ep \right) }  \ \dt\\
&\qquad\qquad \leq C(\delta) \left[ \int_0^\tau
    \mathcal{E}\left( \vr, \vu \ \Big|\ \vr_{NS}, \vc{U}_\ep \right)(t) \
    \dt + \ep |\Ome| \right] + (\ep + \delta) \int_0^\tau \intOe{ |\vu - \vc{U}_\ep |^2 } \ \dt
\end{split}
\]
for any $\delta > 0$. Consequently, in order to conclude, we use the following variant of
\emph{Korn-Poincar\' e inequality}:
\begin{equation} \label{Korn}
\intOe{ \left| \Grad \vc{v} + \Grad^T \vc{v} \right|^2 } \geq C \intOe{ |\vc{v}|^2 },
\end{equation}
with a constant $C$ \emph{independent of} $\ep \to 0$, see Theorem
\ref{TKP} in Section \ref{K}. This allows to estimate
  $\int_0^\tau \intOe{ |\vu - \vc{U}_\ep |^2 }\ \dt$ with $\int_0^\tau
  \intOe{ \left( \tn{S}(\Grad \vu) - \tn{S}(\Grad \vc{U}_\ep ) \right)
    : \left( \Grad \vu - \Grad \vc{U}_\ep \right) }  \ \dt$. For that
to work we had to assume that the bulk viscosity coefficient $\eta$
is strictly positive, since otherwise \eqref{Korn} would need to be
replaced with its {conformal} version \eqref{CKP}.

Finally, as a consequence of a Gronwall-type argument, we obtain:
\begin{equation} \nonumber 
\mathcal{E}_\ep \left( \vr, \vu \ \Big|\ \vr_{NS}, \vc{U}_\ep \right) (\tau) \leq C \left[ \ep |\Ome| +
\mathcal{E}_\ep \left( \vr, \vu \ \Big|\ \vr_{NS}, \vc{U}_\ep \right)
(0) \right]\qquad \forall 0 \leq \tau \leq T,
\end{equation}
completing the proof of Theorem \ref{Tvisc}. \qed

\section{A Korn inequality in thin channels} \label{K}

In this section we discuss various variants of Korn and Korn-Poincar\'
e inequalities that may be of independent interest. In particular, we
show the Korn-Poincar\' e inequality (\ref{Korn}). We assume that:
\begin{equation}\label{CS}
\Omega_\ep = \left\{ x=(\ep \xh, z) \ \Big| \ z\in (0,1), \ \xh\in \omega_h(z)\right\}\subset \mathbb{R}^n,
\end{equation}
where $\{\omega_h(z)\}_{z\in [0,1]}$ is a uniformly Lipschitz family of simply connected bounded domains
$\omega(z) \subset \mathbb{R}^{n-1}$, such that the  boundary of
$\Omega_1$ is Lipschitz. We use the following notation: $\mathrm{sym~}
\mathbb{M} = \frac{1}{2} (\mathbb{M} + \mathbb{M}^t)$ and $\mathrm{skew}~
\mathbb{M} = \frac{1}{2} (\mathbb{M} - \mathbb{M}^t)$ for the
symmetric and the skew-symmetric parts of a given matrix
$\mathbb{M}\in \mathbb{R}^{n\times n}$, and $so(n)$ for the space of all
skew-symmetric matrices $\mathbb{M} = \mathrm{skew}~ \mathbb{M} \in \mathbb{R}^{n\times n}$.

\begin{Theorem} \label{TKP}
Let $\vc{v}\in W^{1,2}(\Ome; \mathbb{R}^n)$ satisfy:
\begin{equation}\label{BC}
\vc{v} \cdot \vc{n} |_{\partial \Ome} = 0, \quad\vc{v}(\xh, z) = 0
  \qquad \forall z \in \{ 0,1 \},\ \xh \in \ep \omega(z).
\end{equation}
Then, we have the following bounds with a constant $C$ independent of $\ep$ and $\vc{v}$:
\begin{equation}\label{ko1}
\intOe{|\Grad \vc{v}|^2} \leq \frac{C}{\ep^2} \intOe{|\mathrm{sym}\Grad \vc{v}|^2}
\end{equation}
\begin{equation}\label{ko2}
\intOe{|\vc{v}|^2} \leq C \intOe{|\mathrm{sym}\Grad \vc{v}|^2}.
\end{equation}
\end{Theorem}


\subsection{An approximation theorem}

Towards the proof of Theorem \ref{TKP}, we first recall the classical Korn's inequality:

\begin{Theorem} \label{korn}
Let $\Omega\subset \mathbb{R}^n$ be an open bounded connected and Lipschitz
domain. For every $\vc{v}\in W^{1,2}(\Omega; \mathbb{R}^n)$ there exists a
matrix $\mathbb{A}\in so(n)$ such that:
\begin{equation}\label{krn}
\intO{|\Grad \vc{v} - \mathbb{A}|^2} \leq C \intO{|\mathrm{sym}\Grad \vc{v}|^2}.
\end{equation}
The constant $C$ above depends only on the domain $\Omega$, but not on
$\vc{v}$. The constant is invariant under dilations of $\Omega$ and it
is uniform for the class of domains that are bilipschitz
equivalent with controlled Lipschitz constants.
\end{Theorem}

It is easy to check that the optimal $\mathbb{A}$ in the left hand side
of (\ref{krn}) equals $\mathbb{A} = \mbox{skew} \fint_\Omega \Grad
  \vc{v}\ \dx$. Armed with this observation, we derive a fine
approximation of $\Grad \vc{v}$ that is suitable for the thin limit problem in
Theorem \ref{TKP}. This approach is motivated by a similar construction in \cite{FJM}.

\begin{Theorem} \label{TK2}
Let $\vc{v}\in W^{1,2}(\Ome; \mathbb{R}^n)$ satisfy the boundary conditions (\ref{BC}). Then, there
exists a smooth mapping $\mathbb{A}: [0,1]\to so(n)$ such that:
\begin{equation}\label{ap1}
\intOe{| \Grad \vc{v}(\xh,z) - \mathbb{A}(z)|^2 } \leq C \intOe{|\mathrm{sym} \Grad \vc{v} |^2 },
\end{equation}
\begin{equation}\label{ap2}
\int_0^1 |\mathbb{A}|^2 \ \dz + \int_0^1|\partial_z\mathbb{A}|^2 \ \dz
\leq \frac{C}{\ep^2} \fint_{\Ome} |\mathrm{sym} \Grad \vc{v} |^2 \ \dx,
\end{equation}
where the constant $C$ above is independent of $\ep$ and $\vc{v}$.
\end{Theorem}

\begin{proof}
{\bf 1.} We identify $\vc{v}$ with its extension on an infinite
curvilinear cylinder as in (\ref{CS}) with $z\in \mathbb{R}$, where we put
$\omega_h (z) = \omega_h (0)$ for $z<0$, $\omega_h(z) = \omega_h(1)$ for $z>1$
and $\vc{v}(\xh, z) = 0$ for $z<0$ and $z>1$, { and $\xh \in \ep \omega (z)$. }
For each $z_0\in \mathbb{R}$, we define the sets:
\[
B_{z_0, \ep} = \left\{x= (\xh, z ) \ \Big| \ z \in (z_0 - \ep, z_0 + \ep), \ \xh \in \ep\omega_h(z) \right\},
\]
and the approximation fields:
\[
\tilde{\mathbb{A}} (z_0) = \fint_{\ep\omega_h(z)} \mathrm{skew}\Grad \vc{v} (\xh, z_0) \ {\rm d}\xh
\quad \mbox{ and } \quad {\mathbb{A}} = \kappa_\ep * \tilde{\mathbb{A}}
\]
by means of a convolution with a regularization kernel $\kappa_\ep = \kappa_\ep(z)$. We set $\kappa_\ep(z)
= \frac{1}{\ep}\kappa(\frac{z}{\ep})$ for some smooth nonnegative
$\kappa\in C^\infty_c$ supported in $(-\frac{1}{2},\frac{1}{2})$ and
with integral $1$. Note that $\mathbb{A}\in C_c^\infty (\mathbb{R}; so(n))$ and in particular:
\begin{equation}\label{zero}
\mathbb{A}(-1) = \mathbb{A}(2) = 0.
\end{equation}

Application of Korn's inequality (\ref{krn}) on sets $B_{z_0, \ep}$ gives:
\begin{equation} \label{uno}
\fint_{B_{z_0,\ep}} | \Grad \vc{v}  - \mathbb{A}_{z_0, \ep} |^2 \ \dx
\leq C \fint_{B_{z_0, \ep}} |\mathrm{sym}\Grad \vc{v} |^2 \ \dx,
\end{equation}
with a uniform constant $C$ (independent of $z_0$, $\ep$ and
$\vc{v}$) and some appropriate $\mathbb{A}_{z_0, \ep} \in so(n)$.
Note that for every $z'\in \mathbb{R}$ we have:
\[
\tilde{\mathbb{A}}(z') - \mathbb{A}_{z_0,\ep} = \fint_{\ep \omega_h(z')}
\mathrm{skew} \Grad \vc{v} - \mathbb{A}_{z_0, \ep} \ {\rm d}\xh
= \fint_{\ep \omega_h(z')} \Grad \vc{v}  - \mathbb{A}_{z_0, \ep} -
\mathrm{sym} \Grad \vc{v} \ {\rm d}\xh.
\]
Using the above for $z'\in(z_0 - \ep, z_0 + \ep)$
we obtain, in view of (\ref{uno}):
\begin{equation}\label{due}
\begin{split}
\left| \mathbb{A} (z'') -  \mathbb{A}_{z_0, \ep}\right|^2 & = \left|
( \kappa_\ep * (\tilde{ \mathbb{A} } - \mathbb{A}_{z_0, \ep}) ) (z'')\right|^2
\leq C \fint_{z''-\ep/2}^{z''+\ep/2} |\tilde{ \mathbb{A} }(z') -
\mathbb{A}_{z_0, \ep} |^2 \ \mathrm{d}z' \\ & \leq C
\fint_{z''-\ep/2}^{z''+\ep/2} \fint_{\ep\omega_h(z')}|\Grad \vc{v} -
\mathbb{A}_{z_0, \ep}  - \mathrm{sym} \Grad \vc{v}|^2 \ \mathrm{d}\xh
\ \mathrm{d}z' \\ & \leq C \fint_{B_{z_0, \ep}} |\Grad \vc{v} -
\mathbb{A}_{z_0, \ep} |^2 + |\mathrm{sym} \Grad \vc{v}|^2\ \dx \\ & \leq C
\fint_{B_{z_0, \ep}} |\mathrm{sym} \Grad \vc{v}|^2\ \dx
\qquad\qquad \forall z''\in (z_0 -\frac{\ep}{2}, z_0 + \frac{\ep}{2}).
\end{split}
\end{equation}
Similarly, we deal with the derivative $\partial_z { \mathbb{A} }$:
\begin{equation}\label{tre}
\begin{split}
\left| \partial_z\mathbb{A} (z'') \right|^2 & = \left|
\partial_z( \kappa_\ep * (\tilde{ \mathbb{A} } - \mathbb{A}_{z_0,
  \ep}) ) (z'')\right|^2 = \left| ((\partial_z\kappa_\ep) * (\tilde{ \mathbb{A} } - \mathbb{A}_{z_0,
  \ep}) ) (z'')\right|^2 \\ & \leq \frac{C}{\ep^2}
\fint_{B_{z_0, \ep}} |\mathrm{sym} \Grad \vc{v}|^2\ \dx
\qquad\qquad \forall z''\in (z_0 -\frac{\ep}{2}, z_0 + \frac{\ep}{2}).
\end{split}
\end{equation}

\smallskip

{\bf 2.} We now estimate, by (\ref{uno}) and (\ref{due}):
\[
\begin{split}
\fint_{B_{z_0, \ep/2}} |\Grad \vc{v}(\xh, z) - \mathbb{A}(z)|^2 \ \dx & \leq C
\Big( \fint_{B_{z_0, \ep/2}} |\Grad \vc{v} - \mathbb{A}_{z_0, \ep}|^2\ \dx +
\fint_{z_0-\ep/2}^{z_0+\ep/2} |\mathbb{A}(z'') - \mathbb{A}_{z_0, \ep}
|^2 \ \mathrm{d}z'' \Big)  \\ &  \leq C \Big( \fint_{B_{z_0, \ep}} |\Grad \vc{v} - \mathbb{A}_{z_0,  \ep}|^2 \ \dx +
\fint_{z_0-\ep/2}^{z_0+\ep/2} \fint_{B_{z_0, \ep}} |\mathrm{sym} \Grad
\vc{v}|^2 \ \dx  \ \mathrm{d}z''\Big) \\ & \leq C \fint_{B_{z_0, \ep}} |\mathrm{sym} \Grad
\vc{v}|^2 \ \dx,
\end{split},
\]
which implies (\ref{ap1}) through an easy covering argument.
Likewise, (\ref{tre}) yields:
\[
\begin{split}
\fint_{z_0-\ep/2}^{z_0+\ep/2} |\partial_z\mathbb{A}(z'')|^2 \ \mathrm{d}z''   \leq \frac{C}{\ep^2}
\fint_{z_0-\ep/2}^{z_0+\ep/2} \fint_{B_{z_0, \ep}} |\mathrm{sym} \Grad
\vc{v}|^2 \ \dx  \ \mathrm{d}z'' = \frac{C}{\ep^2} \fint_{B_{z_0, \ep}} |\mathrm{sym} \Grad
\vc{v}|^2 \ \dx,
\end{split}
\]
and a further covering argument results in:
\[
\int_{-1}^{2}|\partial_z\mathbb{A}|^2 \ \dz \leq \frac{C}{\ep^2} \fint_{\Ome} |\mathrm{sym} \Grad \vc{v} |^2\ \dx.
\]
Using Poincar\'e's inequality to the function $\mathbb{A}$ and noting (\ref{zero}), we finally obtain (\ref{ap2}).
\end{proof}

\subsection{A uniform Poincar\'e inequality for vector fields}

In the proof of Theorem \ref{TKP} we need yet another result, which
is a Poincar\'e inequality for vector fields that are tangent on the
boundary of $\omega_h(z)$ (see \eqref{CS}), and with constant independent of
$z\in[0,1]$. Let us point out that there are many results \cite{2,1,3}
regarding the dependence of $C$ on  an open bounded connected and
Lipschitz set $\Omega\subset \mathbb{R}^n$ in:
\begin{equation}\label{poin}
\int_\Omega |v - \fint_\Omega v |^2 \ \dx \leq C \int_\Omega |\Grad
v|^2 \ \dx \qquad \forall v\in W^{1,2}(\Omega).
\end{equation}
These results are linked to the fact that the smallest $C$ in
(\ref{poin}) is the inverse of the first nonzero eigenvalue
$\lambda_2$ of the Neumann problem for $-\Delta$ on $\Omega$. It is then
known \cite{1}, that $\lambda_2 \geq C_n
\frac{r^n}{\bar r^{n+2}}$ where $r$ and $\bar r$ are the inner and outer radii of
the star-shaped $\Omega$.

Further, in \cite{2} it has been proved that  (\ref{poin}) is valid
with $C$ that is uniform for all $\Omega$  which are uniformly
Lipschitz  with uniformly bounded diameter.
More precisely, $C$ depends only on constants $n$, $\bar r$, $\gamma$ and $M$ below,
for any open and connected $\Omega$ satisfying the following two conditions:
\begin{itemize}
\item[(i)] $\Omega$ is a subset of the ball $B(0,\bar r)\subset \mathbb{R}^n$.
\item[(ii)] At each point $x\in\partial\Omega$ there exists a local
  orthonormal coordinate system such that writing, in this system, $x=(\hat{x}, x^n)$
we have the following. There exists a Lipschitz function $\phi:\hat{
\mathcal{O}} \to \mathbb{R}$ with Lipschitz constant $M$ and we have:
$$ z=(\hat z, z^n)\in\mathcal{O}\cap \Omega \qquad \mbox{if and only
  if} \qquad z\in\mathcal{O} \ \mbox{ and } \ z^n>\phi(\hat z),$$
where we denoted:
\[
\begin{split}
& \hat{\mathcal{O}} = \big\{\hat z\in \mathbb{R}^{n-1} \ \big | \ | (\hat z -\hat x) \cdot e_i |<\gamma \
\mbox{ for all } i=1, \ldots, n-1\big\} \\ &
\mathcal{O} = \big\{z=(\hat z, z^n)\in \mathbb{R}^n \ \big | \ \hat z\in
\hat{\mathcal{O}} \ \mbox{ and } \ |z^n - x^n|<M\gamma\sqrt{n-1}\big\}.
\end{split}
\]
\end{itemize}
Note that boundary of each $\Omega$ as above is uniformly Lipschitz continuous.
This results had been recently extended in \cite{3} to more general
classes of domains, that are uniformly bounded in:
the diameter, the interior cone condition, and an appropriate measure of connectedness.

We now deduce the needed vectorial Poincar\'e inequality:

\begin{Theorem}\label{poin2}
Let $\Omega\subset \mathbb{R}^n$ be an open connected domain satisfying
conditions (i) and (ii) above. Let $\vc{v}\in W^{1,2}(\Omega; \mathbb{R}^n)$ satisfy $
\vc{v} \cdot \vc{n} |_{\partial\Omega} = 0$. Then:
\begin{equation}\label{poin3}
\int_\Omega |\vc{v} |^2 \ \dx \leq C \int_\Omega |\Grad \vc{v}|^2 \ \dx,
\end{equation}
with $C$ independent of $\vc{v}$ and depending on $\Omega$ only
through $n$, $\bar r$, $\gamma$ and $M$.
\end{Theorem}
\begin{proof}
It is easy to note that conditions (i) and (ii) ensure the following uniform bound:
\begin{equation}\label{new}
|\vc{a}|^2 \leq C\int_{\partial\Omega} \left| \vc{a} \cdot \vc{n} \right|^2 \qquad \forall \vc{a} \in \mathbb{R}
^n.
\end{equation}
Indeed, sliding the plane perpendicular to $\vc{a}$ along the
direction $\vc{a}$, at the first point $x\in\partial\Omega$ where this
plane touches the boundary, vector $\vc{a}$ has scalar product bounded
away from zero, with every element of Clarke's subdifferential of $\phi$ at
$\hat x$. Consequently, $C$ in (\ref{new}) depends only on $n$, $\gamma$ and $M$.

Applying (\ref{new}) to the vector $\vc{a}=\fint_\Omega \vc{v}\ \dx$,
we get the following chain of uniform inequalities:
\[
\big|\fint_\Omega \vc{v}\big|^2 \leq C
\int_{\partial\Omega}\big( (\fint_\Omega \vc{v}) \cdot
\vc{n}\big)^2 = C
\int_{\partial\Omega}\big( ( \vc{v} - \fint_\Omega \vc{v}) \cdot \vc{n}\big)^2 \leq C
\int_{\partial\Omega}\big|\vc{v} - \fint_\Omega \vc{v}\big|^2 \leq C \|\vc{v} - \fint_\Omega \vc{v}\|^2_{W^{1,2}(\Omega; \mathbb{R}^n)},
\]
where the last bound follows from the trace theorem
\cite{Adams}. The quoted above result in \cite{2} now implies:
\[
\big|\fint_\Omega \vc{v} \ \dx\big|^2 \leq C \Big( \|\vc{v} -
\fint_{\Omega} \vc{v}\|^2_{L^2(\Omega;\mathbb{R}^n)} + \|\nabla_x \vc{v}
\|^2_{L^2(\Omega;\mathbb{R}^n)}\Big) \le C \int_\Omega |\Grad \vc{v}|^2 \ \dx,
\]
resulting in:
\[
\int_\Omega |\vc{v} |^2 \ \dx \leq 2 \int_\Omega \big|\vc{v} -
\fint_\Omega \vc{v}|^2\ \dx + 2|\Omega|\big|\fint_\Omega \vc{v}\ \dx\big|^2
\leq C \int_\Omega |\Grad \vc{v}|^2 \ \dx
\]
and establishing the proof.
\end{proof}

\begin{Remark} Note that the uniformity assumptions of Theorem \ref{poin2} clearly hold for the
family of cross sections $\{\omega_h(z)\}_{z\in [0,1]}$ because $\Omega_1$ is Lipschitz.
In this case, one can alternatively deduce (\ref{poin3}) by an argument by contradiction that we now sketch.

Assume that there was a sequence $z_k\in [0,1]$, converging to some $z_0$
and such that $\int_{\omega_h(z_k)}|\vc{v}_k|^2 = 1$ but $\int_{\omega_h(z_k)} |\Grad \vc{v}_k|^2\leq 1/k$
for some vector fields $\vc{v}_k\in W^{1,2}(\omega_h(z_k);
\mathbb{R}^{n-1})$ each tangential on the boundary $\partial\omega_h(z_k)$ of its own domain.
The uniform Lipschitz continuity of $\Omega_1$ ensures that
extending $\vc{v}_k$ on the large ball $B = B(0, \bar r)$ that contains all sets $\omega_h(z_k)$, still
obeys the uniform bound $\|\vc{v}_k\|_{W^{1,2}(B;\mathbb{R}^{n-1})}\leq C$.
Thus without loss of generality $\vc{v}_k$ converges to some
$\vc{v}_0$, weakly in $W^{1,2}(B; \mathbb{R}^{n-1})$. Existence of the Lipschitz continuous
homotopy between sets $\omega_h(z_n)$ allows now to deduce that
this implies $\int_{\omega_h(z_0)} |\Grad \vc{v}_0|^2 = 0 $ and
$\vc{v}_0 \cdot \vc{n} = 0$ on
$\partial\omega_h(z_0)$. Consequently, $\vc{v}_0=0$ in $\omega_h(z_0)$,
contradicting  the assumption $\int_{\omega_h(z_k)}|\vc{v}_k|^2 = 1$.
\end{Remark}

\subsection{The proof of Theorem \ref{TKP}}

Let $\mathbb{A}:[0,1]\to so(n)$ be the approximation function in
Theorem \ref{TK2}. Using \eqref{ap1} and \eqref{ap2} we get:
\[
\fint_{\Ome} |\Grad \vc{v} |^2 \ \dx \leq C \Big( \fint_{\Ome} |\Grad
\vc{v}(\xh, z) - \mathbb{A}(z) |^2 \ \dx + \int_0^1 |\mathbb{A}(z)|^2 \ \dz\Big)
\leq \frac{C}{\ep^2} \fint_{\Ome} |\mathrm{sym} \Grad \vc{v} |^2 \ \dx,
\]
which establishes (\ref{ko1}).
Towards proving (\ref{ko2}), define for a smooth curve $X:[0,1]\to \mathbb{R}^{n-1}$ such that
$(X(z),z)\in\Omega_1$ for all $z\in(0,1)$, the following set:
\[
S_{r,\ep} = \left\{ x=(\xh, z) \ \Big | \ z\in (0,1), \ \xh\in \ep  B(X(z), r)\right\}.
\]
Clearly, $S_{r,\ep}\subset\Omega_\ep$ for a sufficiently small $r>0$.
We have the following Poincar\'e inequality:
\[
\fint_{\Omega_1} |{v}|^2 \ \dx \leq C \Big( \fint_{\Omega_1} |\Grad
{v}|^2 \ \dx + \fint_{S_{r,1}} |{v}|^2 \ \dx\Big) \qquad
\forall  v\in W^{1,2}(\Omega_1),
\]
which by an easy scaling argument translates to:
\begin{equation}\label{jeden}
\fint_{\Ome} |{v}|^2 \ \dx \leq C \Big( \ep^2\fint_{\Ome} |\Grad
{v}|^2 \ \dx + \fint_{S_{r,\ep}} |{v}|^2 \ \dx\Big) \qquad
\forall  v\in W^{1,2}(\Ome).
\end{equation}

If additionally the scalar function $v$ obeys: $v(\cdot, 0) = v(\cdot, 1)=0$, then the change of
variables and the Poincar\'e inequality on $[0,1]$ yield:
\begin{equation}\label{dwa}
\begin{split}
\fint_{S_{r,\ep}} |{v}|^2 \ \dx & = \int_0^1 \fint_{\ep B(X(z), r)}|{v}|^2 \ \mathrm{d}\xh \ \dz =
\fint_{\ep B(0, r)} \int_0^1|v(\xh+\ep X(z), z)|^2 \ \dz \
\mathrm{d}\xh \\ & \leq C
\fint_{\ep B(0, r)} \int_0^1|\partial_z v +\ep (\nabla_{\vc{x}_h}v)\partial_z X |^2 \ \dz \ \mathrm{d}\xh
\leq C \fint_{S_{r,\ep}} |\partial_z{v}|^2 + \ep^2 |\Grad v|^2 \ \dx,
\end{split}
\end{equation}
 where $\nabla_{\vc{x}_h}v$ denotes the derivative of $v$ in the horizontal
directions in $\vc{x}_h$. Applying (\ref{jeden}) and (\ref{dwa}) to
$v= \vc{v} \cdot e_n $ results now in the following bound, in view of the already proven (\ref{ko1}):
\begin{equation}\label{trzy}
\begin{split}
\fint_{\Ome}  ( \vc{v} \cdot e_n )^2 \ \dx & \leq C
\Big( \fint_{\Ome} |\partial_z ( \vc{v} \cdot e_n )|^2 \ \dx + \ep^2 \fint_{\Ome}|\Grad v|^2 \ \dx
\Big) \leq C \fint_{\Ome} |\mathrm{sym} \Grad \vc{v} |^2 \ \dx.
\end{split}
\end{equation}

Further, we note that for almost every $z\in [0,1]$, the vector field
$\vc{v} - (\vc{v} \cdot e_n ) e_n\in W^{1,2}(\ep\omega_h(z);
\mathbb{R}^{n-1})$ is tangential on the boundary $\partial(\ep\omega_h(z))$ and
thus we may apply the uniform Poincar\'e inequality in Theorem \ref{poin2} whose constant on the domain
$\ep\omega(z)$ scales like $\ep^2$ with respect to the constant on the domain
$\omega(z)$. Consequently:
\begin{equation}\label{cztery}
\begin{split}
 \fint_{\Ome} |\vc{v} - ( \vc{v}\cdot e_n ) e_n|^2 \ \dx & \leq C
 \int_0^1 \fint_{\ep\omega(z)} |\vc{v} - (\vc{v}\cdot e_n) e_n|^2
 \ \mathrm{d}\xh \ \dz \\ & \leq C\ep^2 \int_0^1 \fint_{\ep\omega(z)} |\nabla_{\xh}\vc{v}|^2
 \ \mathrm{d}\xh \ \dz \leq C \fint_{\Ome} |\mathrm{sym} \Grad \vc{v} |^2 \ \dx,
\end{split}
\end{equation}
where we used (\ref{ko1}) in the last inequality above. Now,
(\ref{cztery}) and (\ref{trzy}) imply (\ref{ko2}) as claimed.
\endproof

\subsection{An optimal Korn inequality for channels with circular cross sections}

Let us point out that the Korn constant in (\ref{ko1})
blows up, in general, at the rate $\frac{C}{\ep^2}$ which is due to a positive measure
set $\mathcal{C}\subset [0,1]$ where each cross section $\omega_h(z)$ with $z\in \mathcal{C}$ has
a rotational symmetry.

\begin{Example}\label{example1}
Given two Lipschitz functions: $X:[0,1]\to \mathbb{R}^{n-1}$ and a positive
$r:[0,1]\to {(0,\infty)}$, let each set $\omega(z)$ be a ball given by:
\begin{equation}\label{CS2}
\omega(z) = B(X(z), r(z))\subset \mathbb{R}^{n-1}.
\end{equation}
For some nonzero function $Q\in W^{1,2}((0,1), so(n-1))$ satisfying
$Q(0) = Q(1) = 0$, consider the following vector fields:
\begin{equation}\label{CS3}
\vc{v}^\ep(\xh, z) = \Big(Q(z)\big(\xh - \ep X(z)\big), 0\Big).
\end{equation}
Note that $\vc{v}^\ep\in W^{1,2}(\Ome; \mathbb{R}^n)$ and it automatically satisfies the boundary conditions (\ref{BC}).
The non-zero entries of the matrix $\Grad \vc{v}^\ep$ are grouped in
its principal minor of dimension $(n-1)$, and its $n$-th column, that are given by:
\[
[\Grad \vc{v}^\ep]_{(n-1)\times(n-1)} = Q(z), \qquad
\partial_z \vc{v}^\ep = \Big((\partial_z Q)\big(\xh-\ep X(z)\big) -\ep Q\partial_zX, 0\Big)
\]
Consequently:
\[
\fint_{\Ome} |\Grad \vc{v}^\ep|^2 \ \dx\geq \fint_{\Ome} |Q(z)|^2 \ \dx \geq c
\quad \mbox{ and } \quad \fint_{\Ome} |\mathrm{sym} \Grad \vc{v}^\ep|^2 \
\dx \leq C\ep^2.
\]
\end{Example}

We will now show that under assumption (\ref{CS2}) the blow-up of
Korn's constant is precisely due to the presence of vector fields
$\vc{v}^\ep$ in Example \ref{example1}. The result below, although not needed for
the fluid dynamics discussion of the present paper, is of independent
interest and should be compared with paper \cite{LM} where
an optimal Korn's inequality was derived for thin $n$-dimensional
shells around a compact boundaryless $(n-1)$-dimensional mid-surface.

\begin{Theorem}\label{opti}
Let $\Omega_\ep$ be as in (\ref{CS}) with $\omega_h(z)$ given in
(\ref{CS2}) by Lipschitz functions: $X:[0,1]\to \mathbb{R}^{n-1}$ and
$r:[0,1]\to {(0,\infty)}$. Define $\vc{v}_Q^\ep$ by (\ref{CS3}), for every $Q\in \mathcal{I}$ where:
\[
\mathcal{I} = \left\{Q\in W^{1,2}([0,1]; so(n-1)) \ \Big| \ Q(0) = Q(1) = 0\right\}.
\]
Let $\alpha\in [0,1)$. Then for every $\vc{v}\in W^{1,2}(\Ome; \mathbb{R}^n)$
satisfying the boundary conditions (\ref{BC}) and:
\begin{equation}\label{CS4}
\left|\int_{\Ome} 
\Grad \vc{v} : \Grad \vc{v}_Q^\ep 
\  \dx\right|
\leq \alpha \|\Grad \vc{v}\|_{L^2(\Ome)} \|\Grad \vc{v}_Q^\ep\|_{L^2(\Ome)} \qquad \forall Q\in\mathcal{I},
\end{equation}
there holds:
\begin{equation}\label{korngood}
\int_{\Ome} |\Grad \vc{v}|^2 \  \dx\leq \frac{C}{1-\alpha^2} \int_{\Ome} |\mathrm{sym}\Grad \vc{v}|^2 \  \dx,
\end{equation}
with a constant $C$ independent of $\vc{v}$, $\ep$ and $\alpha$.
\end{Theorem}
\begin{proof}
{\bf 1.} The angle condition (\ref{CS4}) implies that:
\[
\int_{\Ome} |\Grad \vc{v} |^2 \  \dx \leq \frac{1}{1-\alpha^2} \int_{\Ome} |\Grad \vc{v} - \Grad
\vc{v}_Q^\ep|^2 \ \dx \qquad \ \mbox{for all}\  Q\in\mathcal{I}.
\]
Let $\mathbb{A}:[0,1]\to so(n)$ be as in Theorem \ref{TK2}. Note that,
by construction: $\mathbb{A}(z) = 0$ for $z<-\ep$ and $z>1+\ep$. Thus,
we can modify $\mathbb{A}$ on the intervals $[0,\ep]$ and $[1-\ep, 1]$,
 so that $\mathbb{A}(0) = \mathbb{A}(1)=0$ and (\ref{ap1}),
 (\ref{ap2}) still hold. Define $Q_0(z) = \mathbb{A}_{(n-1)\times (n-1)}(z) \in so(n-1)$ as
the principal minor of $\mathbb{A}(z)$ of dimension $(n-1)$. Then
$Q_0\in\mathcal{I}$ and {using the above we have:}
\begin{equation}\label{CS5}
\begin{split}
(1-\alpha^2)\fint_{\Ome}  |\Grad \vc{v}|^2 \ \dx & \leq { \fint_{\Ome} }|\Grad \vc{v} - \Grad
\vc{v}_{Q_0}^\ep|^2 \ \dx  \\ & \leq C \Big( \fint_{\Ome} |\Grad
\vc{v}(\xh, z) - \mathbb{A}(z)|^2\ \dx + \fint_{\Ome} |\ep\partial_z
Q_0|^2 + |\ep Q_0|^2 \ \dx \\ & \qquad\qquad\qquad\qquad
\qquad\qquad\qquad \quad + \fint_{\Ome} |\partial_z \vc{v}|^2 + |\Grad
( \vc{v} \cdot e_n ) |^2 \ \dx\Big) \\ & \leq C \Big( \fint_{\Ome} |\mathrm{sym}\Grad
\vc{v}|^2\ \dx + \ep^2 \int_0^1 |\partial_z \mathbb{A}|^2 +
|\mathbb{A}|^2 \ \dz + \fint_{\Ome} |\partial_z \vc{v}|^2 \ \dx\Big)
\\ & \leq C \Big( \fint_{\Ome} |\mathrm{sym}\Grad
\vc{v}|^2\ \dx + \fint_{\Ome} |\partial_z \vc{v}|^2 \ \dx\Big),
\end{split}
\end{equation}
where we applied Theorem \ref{TK2}. We now observe that the last term above satisfies:
\[
\fint_{\Ome} |\partial_z \vc{v}|^2 \ \dx \leq C \Big( \fint_{\Ome} |\partial_z \vc{v} - \mathbb{A}e_n|^2\ \dx +
\fint_{\Ome} |\mathbb{A}e_n|^2\ \dx\Big),
\]
and thus (\ref{CS5}) yields:
\begin{equation}\label{CS6}
\fint_{\Ome} |\Grad \vc{v}|^2 \ \dx \leq C \Big( \fint_{\Ome} |\mathrm{sym}\Grad \vc{v}|^2\ \dx +
\int_0^1 |\mathbb{A}e_n|^2\ \dz \Big).
\end{equation}

\smallskip

{\bf 2.} For each integral term of the form $\int_0^1 (
\mathbb{A}e_n \cdot e_i )^2\ \dz,$ $i=1\ldots (n-1)$, we recall the Hilbert
space identity $\|a\|^2 + \|b\|^2 = \|a-b\|^2 + 2  a \cdot b $
to estimate:
\begin{equation}\label{CS7}
\begin{split}
\int_0^1 ( \mathbb{A}e_n \cdot e_i )^2\ \dz  \leq \frac{C}{\ep^{2(n-1)}}  \Bigg(
\int_0^1&\Big|\partial_z\big(\int_{\ep\omega(z)} ( \vc{v} \cdot e_i ) \ \mathrm{d}\xh\big) -
\int_{\ep\omega(z)} ( \mathbb{A}e_n \cdot e_i ) \
\mathrm{d}\xh\Big|^2\ \dz \\ & +
\int_0^1 \big( \int_{\ep\omega(z)} ( \mathbb{A}e_n \cdot e_i ) \ \mathrm{d}\xh\big)
\partial_z \big(\int_{\ep\omega(z)} ( \vc{v} \cdot e_i ) \ \mathrm{d}\xh\big) \ \dz\Bigg).
\end{split}
\end{equation}
Using Reynolds transport theorem, we find the derivative:
\[
\partial_z \big(\int_{\ep\omega_h(z)} ( \vc{v} \cdot e_i ) \
\mathrm{d}\xh\big) = \int_{\ep\omega_h(z)} ( \partial_z\vc{v} \cdot 
e_i ) + \mbox{div}\Big( ( \vc{v}\cdot  e_i ) \partial_z\phi^\ep(\xh,z)\Big) \ \mathrm{d}\xh
\]
in terms of the derivative $\partial_z\phi^\ep$ of the flow of
diffeomorphisms $\phi^\ep(\cdot, z):\ep\omega_h(0)\to \mathbb{R}^{n-1}$ such that
$\phi^\ep(\ep\omega(0), z) = \ep\omega_h(z)$. In fact, we can take $\phi^\ep(\xh, z)
= \ep\phi^1(\frac{1}{\ep}\xh, z)$, whereas the simple form of the
cross sections in (\ref{CS2}) ensures that:
\[
\phi^\ep(\xh, z) = \ep\partial_z X(z) +\frac{\partial_z r(z)}{r(0)}\xh
\qquad \forall \xh\in B(\ep X(z), \ep r(0)), \quad z\in [0,1],
\]
so that $\partial_z\phi^\ep(\xh, z) = \ep\partial^2_zX +\frac{\partial^2_zr(z)}{r(0)}\xh$.

Thus, we may bound the first term in the right hand side of (\ref{CS7}) by:
\begin{equation}\label{CS8}
\begin{split}
\frac{1}{\ep^{2(n-1)}} \int_0^1\Big|\partial_z&\big(\int_{\ep\omega_h(z)} ( \vc{v} \cdot e_i ) \ \mathrm{d}\xh\big) -
\int_{\ep\omega_h(z)} ( \mathbb{A}e_n \cdot e_i ) \
\mathrm{d}\xh\Big|^2\ \dz \\ & \leq C
\int_0^1 \left|\fint_{\ep\omega_h(z)} ( \partial_z\vc{v} - \mathbb{A}e_n ) \cdot
e_i  + \mbox{div}\Big( ( \vc{v}\cdot e_i ) \partial_z\phi^\ep(\xh,z)\Big) \ \mathrm{d}\xh
\right|^2 \ \dz \\ & \leq C\Big(\fint_{\Ome} |\partial_z\vc{v} - \mathbb{A}e_n|^2 \ \dx + \ep^2 \fint_{\Ome}
|\Grad ( \vc{v} \cdot e_i ) |^2 \ \dx + \fint_{\Ome} ( \vc{v} \cdot e_i )^2 \ \dx
\Big) \\ & \leq C \fint_{\Ome} |\mathrm{sym}\Grad \vc{v}|^2\ \dx,
\end{split}
\end{equation}
where the last inequality above follows from Theorem \ref{TK2} and Theorem \ref{TKP}.
For the second term in (\ref{CS7}), we integrate by parts to get:
\begin{equation}\label{CS9}
\begin{split}
\frac{1}{\ep^{n-1}} \Big |\int_0^1 ( \mathbb{A}e_n \cdot e_i )
\partial_z &\big(\int_{\ep\omega(z)} ( \vc{v} \cdot e_i ) \ \mathrm{d}\xh\big) \ \dz\Big|
 \leq C \left|\int_0^1 (\partial_z ( \mathbb{A}e_n \cdot e_i ) )
\big(\fint_{\ep\omega_h(z)} ( \vc{v} \cdot e_i ) \ \mathrm{d}\xh\big) \ \dz\right |
\\ & \leq \left(\int_0^1 |\partial_z \mathbb{A}|^2 \ \dz\right)^{1/2}
\left(\fint_{\Ome} ( \vc{v} \cdot e_i ) \ \dx\right)^{1/2} \\ & \leq C
\left(\fint_{\Ome} |\mathrm{sym}\Grad \vc{v}|^2\ \dx\right)^{1/2} \left(\fint_{\Ome} |\Grad \vc{v}|^2\ \dx\right)^{1/2},
\end{split}
\end{equation}
using Theorem \ref{TK2} and (\ref{cztery}).

\smallskip

{\bf 3.} Finally, (\ref{CS6}), (\ref{CS7}), (\ref{CS8}) and (\ref{CS9}) imply:
\[
\fint_{\Ome}  |\Grad \vc{v}|^2\ \dx \leq \frac{C}{1-\alpha^2} \fint_{\Ome} |\mathrm{sym}\Grad \vc{v}|^2\ \dx +
\frac{C}{1-\alpha^2} \left(\fint_{\Ome} |\mathrm{sym}\Grad \vc{v}|^2\
  \dx\right)^{1/2} \left(\fint_{\Ome} |\Grad \vc{v}|^2\ \dx\right)^{1/2},
\]
which yields (\ref{korngood}) and achieves the proof.
\end{proof}

\begin{Remark}
It would be interesting to prove an optimal Korn's inequality in
the spirit of Theorem \ref{opti}, for the general case of thin channels as in
(\ref{CS}). A natural candidate for the functional kernel
$\mathcal{I}$ in (\ref{CS4}) is then the following space:
\[
\begin{split}
\mathcal{I} = \Big\{(Q, X)\in W^{1,2}([0,1]; & so(n-1)\times \mathbb{R}^{n-1}) \
  \Big| \ (Q, X)(0) = (Q,X)(1) = 0 \quad \mbox{and}\\ &
\xh\mapsto Q(z)\xh + X(z) \mbox{ is tangent on } \partial\omega(z),
\mbox{ for a.a. } z\in[0,1] \Big\}.
\end{split}
\]
Each element $(Q,X)\in\mathcal{I}$ generates a vector field
$\vc{v}_{Q,X}^\ep \in W^{1,2}(\Ome; \mathbb{R}^n)$ satisfying (\ref{BC}), where we set:
$$\vc{v}_{Q,X}^\ep (\xh, z) = \Big(Q(z)\xh + \ep X(z), 0\Big).$$

Note that if $\omega(z)$ has no rotational symmetry, then automatically $(Q, X)(z) = 0$.
Further, observe that every closed set $\mathcal{C}\subset[0,1]$ is the locus of
rotationally symmetric sections in some smooth channel $\Omega_1$. Namely, let
$r:[0,1]\to R$ be a smooth nonnegative function such that $r^{-1}(0) =
\mathcal{C}$. Let $\omega_0$ be a smooth domain with no rotational
symmetry, satisfying $B(0,1)\subset\omega_0\subset B(0,2)\subset \mathbb{R}^{n-1}$. Define:
$$\omega(z) = B(0,1) \cup \Big\{\Big(1+ {r}(z)(|\xh|-1)\Big)\xh \ \Big| \ \xh\in
\omega_0, \ |\xh|\geq 1\Big\}.$$
Then $\omega(z)$ equals $B(0,1)$ for all $z\in \mathcal{C}$, and otherwise
$\omega(z)$ has no rotational symmetry, so that:
\begin{equation}\label{korn3}
\int_{\omega(z)} |\Grad \vc{u}|^2 \ \dx \leq C \int_{\omega(z)} |\mathrm{sym}\Grad \vc{u}|^2 \ \dx,
\end{equation}
valid for all $\vc{u}\in W^{1,2}(\omega(z); \mathbb{R}^{n-1})$
tangent on $\partial\omega(z)$ and all $z\not\in \mathcal{C}$,
with a uniform $C$.

Observe now that taking the set $\mathcal{C}$ nowhere dense implies that $\mathcal{I}=\{0\}$,
indicating that (\ref{korngood}) holds for all $\vc{v}$ satisfying
(\ref{BC}) (here $\alpha=0$). On the other hand, if $\mathcal{C}$ has positive
measure (as valid for the ``fattened'' Cantor set), then Korn's inequality (\ref{korn3}) still fails at all
$z\in \mathcal{C}$.
\end{Remark}

\def\cprime{$'$} \def\ocirc#1{\ifmmode\setbox0=\hbox{$#1$}\dimen0=\ht0
  \advance\dimen0 by1pt\rlap{\hbox to\wd0{\hss\raise\dimen0
  \hbox{\hskip.2em$\scriptscriptstyle\circ$}\hss}}#1\else {\accent"17 #1}\fi}

\end{document}